\documentclass[10pt]{article}
\usepackage{lineno}
\usepackage{amssymb}
\usepackage{etoolbox}
\usepackage{amsmath}
\usepackage{amsthm}
\usepackage{hyperref}
\usepackage{amsfonts}
\usepackage{color}
\usepackage{xcolor}
\usepackage{epsfig}
\usepackage{booktabs}
\usepackage{graphics,subfigure} 
\usepackage{graphicx}
\usepackage{float}
\usepackage{epsf}
\usepackage{bm}
\usepackage{calc}
\usepackage{mathrsfs}
\usepackage[numbers,sort&compress]{natbib}
\usepackage{indentfirst} 
\definecolor{myGreen}{rgb}{0.9, 0.93,0.9}
\newcommand*\linenomathpatch[1]{%
	\cspreto{#1}{\linenomath}%
	\cspreto{#1*}{\linenomath}%
	\csappto{end#1}{\endlinenomath}%
	\csappto{end#1*}{\endlinenomath}%
}

\linenomathpatch{equation}
\linenomathpatch{gather}
\linenomathpatch{multline}
\linenomathpatch{align}
\linenomathpatch{alignat}
\linenomathpatch{flalign}

\newtheoremstyle{thmm}{1.5ex plus 1ex minus .2ex}{1.5ex plus 1ex minus
	.2ex}{\rmfamily}{}{\bfseries}{}{1em}{} \theoremstyle{thmm}
\newtheorem{theorem}{Theorem}[section]
\newtheorem{lemma}{Lemma}[section]

\allowdisplaybreaks
\begin{document}
	\date{\today}
	\title{\bf A new second-order consistent splitting scheme for the Natural Convection equations\thanks{This work is supported by the Natural Science Foundation of Henan Province (No. 252300421989) and the Innovative Research Team of Henan Polytechnic University, China (No. T2024-4).}}
	\author{Zhiyong Si
		\footnote{School of Mathematics and Information Science, Henan Polytechnic University, 454003, Jiaozuo, P.R. China. {\tt sizhiyong@hpu.edu.cn (Z.Y. Si)}
		}\, ,\, Yimei Ma\footnotemark[2], Yunxia Wang\footnotemark[2]}
	\maketitle 
	\begin{abstract}
		A novel second-order consistent splitting scheme is proposed for the natural convection equations. The scheme is constructed using Taylor expansions about the time level $t^{n+k}$, where $k$ is a parameter to be determined. It is proved to be stable for all 
		$k>4$, with the present study focusing primarily on the case $k=5$. Compared with the conventional splitting scheme based on Taylor expansions about $t^{n+1}$, the proposed scheme exhibits improved stability. By employing the Sobolev inequality and Gronwall's lemma, we rigorously establish stability of the proposed scheme and derive error estimates in both two and three spatial dimensions. Finally, several numerical experiments are presented to verify the stability and accuracy of the proposed scheme and to demonstrate its effectiveness.

	\end{abstract}
	\section{Introduction}
	\setcounter{equation}{0}
	In this paper, we aim to construct a second-order consistent splitting scheme for the governing equations of natural convection, in which the effects of body forces and heat sources are taken into account.
	\begin{eqnarray}\label{1.1}
		\left\{\begin{array}{lll} 
			\bm{u}_t - \nu\Delta\bm{ u} + (\bm{u} \cdot \nabla) \bm{u} + \nabla p+Gr\nu^{2}g\theta= \bm{f}, \\
			\theta_t  - \lambda\nu\Delta \theta + (\bm{u} \cdot \nabla)\theta= b, \\	 \nabla \cdot\bm{u} = 0. \\
		\end{array}\right.
	\end{eqnarray}
	With the initial conditions:
	\begin{align}
		\bm{u}(x, 0) = \bm{u}_0(x),   \theta(x, 0) = \theta_0(x),   \text{for } x \in \Omega
	\end{align}
	and the boundary conditions:
	\begin{align}
		\bm{u}(x,t)=0, \theta(x,t)=0,  \text{for }(x,t)\in\partial Q_T,
	\end{align}
	where {  \( Q_T = \Omega \times (0, T] \)} with \( \Omega \subset \mathbb{R}^d \) ($d=2,3$) possessing a Lipschitz-continuous boundary \( \partial\Omega \) and \( T > 0 \) is a fixed constant. Here, \( \bm{u} \) denotes the velocity field, \( p \) represents the pressure, and \( \theta \) stands for the temperature; \( \bm{f} \) and \( b \) are the known body force and known heat sources, respectively; \( \nu \) denotes the kinematic viscosity, \( Gr \) is the Grashof number, \( \lambda = Pr^{-1} \) (where \( Pr \) denotes the Prandtl number), and \( g \) is the gravitational acceleration.
	
	Natural convection describes fluid flow driven by temperature differences. The system of natural convection equations comprises a coupled system of Navier-Stokes equations and the energy equation, involving not only the velocity and pressure fields but also the temperature field. It has numerous applications in industry and engineering, such as solar heat absorbers, heat dissipation in electronic devices, and building insulation. Additionally, Equation (1.1) serves as the fundamental governing equations underlying phenomena including  cooling of electronic devices, atmospheric fronts, solar collectors, katabatic winds, atmospheric dynamic systems, natural ventilation and nuclear reactors, and thermal insulation of double-glazed windows \cite{A2011}. Obtaining exact solutions to these equations is highly challenging, prompting a growing number of scholars to investigate their numerical solutions. Numerous studies have been devoted to developing effective methods for natural convection problems (\cite{Li2004,C2006,Wang2006,Luo2003,Huang2012,Huang20122,T1999} and the references therein). However, the numerical solution of natural convection equations is still a challenging undertaking, stemming from their strong nonlinear terms and inherent strong coupling properties.  
	
	In recent years, finite element methods have received increasing attention. In \cite{FU1994,XIN1989}, several finite difference schemes were discussed under the assumption that temperature is treated as a constant, thereby reducing the natural convection equations to the Navier-Stokes equations. For time-dependent conduction-convection problems, Si and Wang \cite{Si2015} integrated the projection method with the modified characteristics-based finite element approach, thereby developing a modified characteristics projection finite element method. In addition, Si et al. \cite{Si2014} proposed a modified characteristics Gauge-Uzawa finite element method (MCGUFE) and established its stability and corresponding error analysis. In \cite{Han2019}, a weak Galerkin finite element method was proposed, and an unconditionally convergent iterative algorithm was presented. Based on mixed finite elements, Luo and Zhu \cite{LUO2003} developed a lowest-order finite difference scheme. In \cite{Si2011}, a semi-discrete defect-correction mixed finite method was considered. In \cite{Si2012}, for the non-stationary conduction-convection problems, Si at el. presented a fully discrete deffect-correction mixed finite element.  Boland and Layton \cite{J1990} studied free and natural convection problems via finite element approximations and provided error analyses. In \cite{LYF2018}, Lei et al. presented a penalty finite element method for two-dimensional stationary conduction-convection problems.  Liang and Zhang \cite{Liang2019} solved a nonlinear natural convection problem on a coarse grid and a linear one on a fine grid. For steady-state natural convection equations, Cibik and Kaya \cite{A2011} formulated a stabilized finite element technique based on projection. Other methods have also been explored, including the defect-correction method \cite{Yun2014} and variational multiscale methods \cite{Yun2013}. Huang et al. \cite{Huang2015} discussed several iterative schemes and established their stability and convergence.    
	
	In recent years, many researchers have begun exploring new iterative methods for time discretization. Rashidi et al. \cite{Rashidi2013} solved the natural convection equations using a multi-step differential transform method. Subsequently, a time discretization method based on an explicit multistage Runge-Kutta scheme was discussed in \cite{Y1999}. He \cite{H2013} applied implicit and explicit Euler iterative schemes to stationary Navier-Stokes equations based on mixed finite elements. Codina \cite{Codina} and Carey \cite{Carey} proposed several iterative methods for 2D steady Navier-Stokes equations.  Guermond and Shen \cite{Gu2003}  conducted a stability analysis for the first-order consistent semi-discrete splitting scheme of the Navier-Stokes equations, which is dependent of time. Additionally, for the first-order consistent splitting scheme, Liu and Pego \cite{J2007,Liu2010} discussed its local-in-time error estimates and stability. Since then, many scholars have attempted to extend these methods to second-order schemes. For instance, a second-order consistent splitting scheme  was introduced in \cite{H2004}, which is primarily designed for Navier-Stokes equations; however, whether this second-order scheme is unconditionally stable remains an open problem. The major difficulty associated with the pressure term is its second-order extrapolation, a problem the third-order pressure-correction scheme also addresses\cite{J1993}. Thus, scholars are seeking new approaches to address this limitation in second-order schemes. J. Girault and Shen \cite{J2003} developed a new category of splitting schemes applicable to incompressible flows, and established the stability properties of the  first-order consistent splitting scheme, which is semi-discrete. Furthermore, Shen et al. \cite{Shen2023} concentrated their efforts on the Navier-Stokes equations, where they performed a comprehensive stability and error analysis for a consistent second-order splitting scheme. 
	
	It is worth noting that both the Adams-Bashforth extrapolation method and the conventional backward differentiation formula (BDF) are both founded on Taylor expansions at  \( t^{n+1} \), which makes the pressure term difficult to tackle because of its second-order extrapolation. Motivated by \cite{Shen2023}, we introduce an undetermined parameter \( k \) and employ Taylor expansions at time \( t^{n+k} \) to develop a new BDF scheme for the natural convection equations. As indicated in \cite{Shen2023}, with an increase in \( k \), the convergence domain of the scheme expands, while the truncation error increases slightly. Therefore, selecting an appropriate \( k \) can enhance the scheme’s stability and improve error estimation.

	What we present in this work is a new second-order consistent splitting scheme, specifically tailored to the natural convection equations governed by no-slip boundary conditions. To be noticed that we perform a rigorous analysis to prove its stability and estimate in two-dimensional (2D) and  three-dimensional (3D) spaces. Our main contributions are summarized as follows:
	
	\indent$\bullet$ First, derived from Taylor expansions at time \( t^{n+5} \), we propose a consistent splitting scheme for the natural convection equations, which is second-order. In contrast to traditional second-order schemes, the presented scheme achieves superior stability.
	
	\indent$\bullet$ Second, we conduct a rigorous proof of the scheme’s  stability in 2D and 3D, leveraging the Sobolev inequality and Gronwall’s lemma. Concurrently, corresponding  error estimates are established for the 2D and 3D cases, respectively.
	
	\indent$\bullet$ Third, we implement a set of numerical cases to evaluate the performance of the newly proposed second-order consistent splitting scheme for the natural convection equations. 
	
	To the best of our knowledge, this paper represents the first time the stability and error analysis of a second-order consistent splitting scheme for the natural convection equations has been rigorously established. In this work, the proposed new scheme, which is constructed from Taylor expansions at time \( t^{n+5} \)-demonstrates enhanced stability.

	The rest of this paper is organized as follows: In the following section outlines the Sobolev spaces and useful lemmas that are utilized throughout the present work. In Section 3, we propose a consistent second-order splitting scheme for the natural convection equations. The scheme is derived using Taylor expansions about the time level $t^{n+k}$, where $k$ is a parameter to be determined. We establish its  stability for $k>4$ and focus primarily on the representative case $k=5$.In Section 4, we analyze the error estimate of the scheme in two dimensions and three dimensions. Finally, in Section 5, several numerical examples are conducted to verify the effectiveness and accuracy of the proposed scheme.
	
	\section{Preliminaries}
	\setcounter{equation}{0}
	This section sets out the notations that are utilized consistently throughout the paper. We denote the inner product by $(\cdot, \cdot)$, and $\Vert \cdot \Vert$ denotes the norm in $L^2(\Omega)$. 	At all the first, we introuce the following Sobolev spaces:
	\begin{align*}
		\bm{H}^{m}(\Omega) = & \left\{ \bm{u} \in  L^2(\Omega) \, ; \, \|\bm{u}\|_{\bm{H}^{m}(\Omega)} < \infty \right\},
	\end{align*}
	equipped with the norm that:
	\begin{align*}
		\|\bm{u} \|_{m(\Omega)} =	\left( \sum_{|\alpha| \leq m} \displaystyle\int_\Omega \left| \bm{D}^\alpha \bm{u} \right|^2 \, dx \right)^{1/2}.
	\end{align*} 
	Moreover, we define the following spaces:
	\begin{align*}
	\bm{X} &:= \bm{H}_0^1(\Omega) = \left\{ \bm{w} \in \bm{H}^1(\Omega) : \left. \bm{w} \right|_{\partial \Omega} = \mathbf{0} \right\},{ \bm{H}:=[H^2(\Omega)\cap H_0^1(\Omega)]^d},\\
	\bm{V}&:= \left\{ \bm{w} \in \bm{X} : \nabla \cdot \bm{w} = 0 \ \text{in} \ \Omega \right\}, 
	M := L_0^2(\Omega) = \{ q \in L^2(\Omega) : \int_{\Omega} q \, d\mathbf{x} = 0 \}.
	\end{align*}
	{ The space $\bm{V}$ is introduced only to state the classical properties
	of the trilinear form. The subsequent analysis does not rely on the
	divergence-free condition.} An arbitrary Banach space is denoted by $\mathcal{V}$ and the notation $L^p(0, T; \mathcal{V})$ is adopted for the standard Lebesgue space. Similarly, the space of continuous functions $C([0, T]; \mathcal{V})$ follows the standard convention. Vectors and vector spaces are distinguished by boldface font. We use $C$ to represent an arbitrary positive constant, which does not depend on the discretization parameters.
	For notational brevity, we often denote $\bm{u}(\bm{x}, t)$ as $\bm{u}(t)$, omitting the exact solution’s spatial dependence
	We define the continuous bilinear forms $a(\cdotp,\cdotp)$ and $d(\cdotp,\cdotp)$ on $\bm{X} \times\bm{X}$ and $\bm{X}\times M$, as in\cite{M1989} :
	\begin{align*}
		a(\bm{u}, \bm{v}) = & \nu(\nabla \bm{u}, \nabla\bm {v}),   \forall \bm{u}, \bm{v }\in \bm{X}, \\
		d(\bm{v}, q) = & (q, \nabla \cdot \bm{v}),   \forall \bm{v} \in \bm{X}, \ \forall q \in M,
	\end{align*}
 and the trilinear form:
$b(\bm{u,v,w})=\int_\Omega(\bm{u}\cdot\nabla)\bm{v}\cdot\bm{w}d\bm{x}$. 
\begin{lemma}[\cite{He2009}]\label{2.1'}
	The following properties are satisfied by the trilinear form $b$:
	\begin{align}
		&b(\bm{u,v,w})=-b(\bm{u,w,v}), { \forall \bm{u,}\in 	\bm{V},\bm{v,w}\in 	\bm{X},}
	\end{align}
	which implies that 
	\begin{align}\label{2.2}
		&	b(\bm{u,v,v})=0, { \forall \bm{u}\in 	\bm{V},\bm{v}\in 	\bm{V}}.
	\end{align}
	By using H\"{o}lder and Sobolev inequalities, for all {  $\bm{u, w} \in\bm{X}, \bm{v}\in \bm{H}$},   there holds that \cite{R1983}
	\begin{align}\label{2.3}
		& b(\bm{u,v,w})\leqslant \Vert\nabla\bm { u}\Vert^{1/2}\Vert \nabla\bm{v} \Vert^{1/2}\Vert\Delta\bm { v} \Vert_{}^{1/2}\Vert\bm {u} \Vert^{1/2}\Vert\bm {w} \Vert,  d=2
	\end{align}  
	\begin{align}\label{2.4}
		b(\bm{u,v,w})\leqslant \Vert\Delta\bm{ v} \Vert_{}^{1/2}\Vert\nabla\bm{v} \Vert^{1/2}\Vert\nabla\bm {u}\Vert\|\bm {w} \Vert,  d=3
	\end{align}
	Additionally, these inequalities are frequently employed in our subsequent analysis\cite{R1983}:
	\begin{align} \label{2.5}
		b(\bm{u, v, w}) \leq \|\bm{u}\|_2^2\|\bm{v}\|_1\|\bm{w}\|_0, \quad{ \forall \bm{u}\in \bm{H}, \bm{v,w} \in\bm{X}}, d=2,3
	\end{align}
	\begin{align}\label{2.6}
		b(\bm{u, v, w}) \leq \|\bm{u}\|_1^2\|\bm{v}\|_2\|\bm{w}\|_0, \quad{ \forall \bm{v}\in \bm{H},  \bm{u,w} \in\bm{X}},d=2,3.
	\end{align}
\end{lemma}
	In \cite{J2007}, some properties of the Stokes pressure were introduced, which we will use in the process of establishing the stability and error estimate. The $\mathcal{P}$  represents Leray-Helmholtz projection  defined on the $\bm{V}$ with vanishing normal components. Based on that we get the definition of the Stokes pressure $\mathnormal{p}_s = P_s(\bm{u})$, which is as follows:
	\begin{align}
		\nabla\mathnormal{p}_s(\bm{u}) =(\Delta \mathcal{P}-\mathcal{P}\Delta) \bm{u},\quad \forall\ \bm{u} \in \bm{H}^2(\Omega, \mathbb{R}^N),
	\end{align}
	This projection underpins the Helmholtz decomposition of $\bm{u}$, which takes the form $\mathnormal{\bm{u}}=\mathcal{P}\bm{u}+\nabla\varPhi$.
	A key orthogonality property of the projection is given by
	\begin{align}
		(\mathcal{P}\bm{u},\nabla q)=(\bm{u}-\nabla \varPhi,\nabla q)=0    \quad\forall q\in \mathit{H}^1(\Omega).
	\end{align}
	A complete proof of the lemma stated below can be found in the reference \cite {J2007} .
	\begin{lemma}[\cite {J2007}]\label{Lem2.2}
		{\it We consider $\Omega \subset \mathbb{R}^N$ ($N \geq 2$) which is a bounded domain with a $C^3$-smooth boundary. $\forall \varepsilon > 0$,  a nonegative constant $C$ can be found such that the following estimate is established for all vector fields $\bm{u} \in \bm{H}^2(\Omega, \mathbb{R}^N) \cap \bm{H}_0^1(\Omega, \mathbb{R}^N)$,
			\begin{align}
				&\| (\Delta \mathcal{P}-\mathcal{P}\Delta) \bm{u}\|^2 \leq (\frac{1}{2}+\varepsilon) \|\Delta \bm{u}\|^2 +C\|\nabla \bm{u}\|^2.
		\end{align}  }
	\end{lemma}
	The following are two Gronwall lemmas that will be utilized in the subsequent analysis.
	\begin{lemma}[discrete Gronwall lemma 1 \cite{J2012}]\label{Lem2.3}
			{\it Let $g^n,h^n,b^n,f^n$ be nonnegative sequences  such that $\delta tg^n<1$ for all $n$  satisfy the inquality
			\begin{align}
				&g^N+\delta t\sum_{n=0}^N h^n\leq\delta t\sum_{n=0}^{N}(b^ng^n+f^n)+B,  
				\quad \forall0\leq N \leq T/\delta t,
			\end{align}
			and suppose the boundedness condition $ \delta t\sum_{n=0}^{T/\delta t}b^n \leq M$ holds. then the following estimate is valid
			\begin{align}
				&g^N+\delta t\sum_{n=0}^N h^n\leq \exp(\sigma M)(B+\delta t\sum_{n=0}^{N}f^n), \quad \forall N\leq T/\delta t,
			\end{align}  
			where we denote $\sigma=\max_{0\leqslant n\leqslant T/\delta t}(1-\delta t b^n)^{-1}$.}
	\end{lemma}
	\begin{lemma}[discrete Gronwall lemma 2 \cite{JS1990}]
		{\it Let $a_n,b_n,c_n,d_n$ be four nonnegative sequences and let  $C$ and $\delta t$ be positive constant such that the following inequality holds:
			\begin{align}
				&a_N+\delta t\sum_{n=1}^N b_n\leq\delta t \sum_{n=0}^{N-1}a_nd_n+\delta t\sum_{n=0}^{N-1} c_n+C, \quad\forall N\geq1,
			\end{align}
			then the following estimate holds:
			\begin{align}
				&a_N+\delta t\sum_{n=1}^N b_n\leq \exp(\delta t\sum_{n=0}^{N-1} d_n)(C+\delta t\sum_{n=0}^{N-1}c_n),  \quad\forall N\geq 1 .
		\end{align}  }
	\end{lemma}
	For the three-dimensional case, the following lemma is critical to our proofs of weak stability and local-in-time error estimates.
	\begin{lemma}[\cite{J2010}]\label{Lem2.5}
		{\it For any given  \(T^*\) and $T^*\in(0,\int_{M}^{\infty} \frac{dz}{\Psi(z)})$, where $\Psi$  be a continuous and increasing function, which is defined from $(0, \infty) $ to $ (0, \infty)$ and let \(M \) be a positive constant.  
			Suppose that quantities \(a_n, w_n \geq 0\) and a positive constant  \(G_* \) that is independent of the time step  \(\tau > 0\) satisfy 
			\[
			a_n + \sum_{i = 0}^{n - 1} \tau \, w_i \leq y_n , \quad\forall 0\leq n \leq n_*,
			\]
			where we denote $y_n := M + \sum_{i = 0}^{n - 1} \tau \, \Psi(a_i)$
			with \(n_* \tau \leq T^*\). Then \(y_n \leq G_*\).} 
	\end{lemma}
	
	\section{A New Second-Order Consistent Splitting Scheme for Natural Convection Equations And Its Stability}
	\setcounter{equation}{0}
	We propose a consistent splitting scheme, which is second-order accurate, for the natural convection equations in this section, and provide a rigorous proof of its  stability in 2D and  3D.
	
	\subsection{A new second-order scheme for natural convection equations}
	
Derived from Taylor expansions at time $ t^
{n+k}$, where $k$ is any positive integer, we propose a consistent second-order splitting scheme for the natural convection equations. Appropriately selecting \( k \) allows us to establish the new scheme's unconditional stability. As noting in \cite{Shen2023}, with an increase in \( k \), the convergence domain of the scheme expands, while its truncation error increases slightly. Given a fixed positive integer \( N \) define the time step as \( \delta t = T/N \); here, \( t^n = n\delta t \) (for \( n = 1, 2, \dots, N \)) and denote the discrete time levels. By expanding \( \varPsi(t^{n+1}) \), \( \varPsi(t^n) \), and \( \varPsi(t^{n-1}) \) we can get the following equations by using the Taylor expansion at \( t^{n+k} \).
	
	\begin{align}
		&\frac{(2k+1)\varPsi(t^{n+1})-4k\varPsi(t^n)+(2k-1)\varPsi(t^{n-1})}{2\delta t}=\varPsi^{'}(t^{n+k})+\frac{1-3k^2}{6}\varPsi^{'''}(t^{n+k})\delta t^2+O(\delta t^3).
	\end{align}
	By expanding $\varPsi(t^{n+1})$ and $\varPsi(t^n)$ via the Taylor expansion at $t^{n+k}$, we have
	\begin{align}
		k\varPsi(t^{n+1})-(k-1)\varPsi(t^n)=\varPsi(t^{n+k})-\frac{k(k-1)}{2}\varPsi^{''}(t^{n+k})\delta t^2+O(\delta t^3).
	\end{align}
	Hence, in this paper, we adopt the following approximation of $\varPsi^{'}(t^{n+k})$ and $\varPsi(t^{n+k})$:
	\begin{align}
		&\varPsi^{'}(t^{n+k})\approx\frac{(2k+1)\varPsi(t^{n+1})-4k\varPsi{t}^n+(2k-1)\varPsi(t^{n-1})}{2\delta t},\\
		&\varPsi(t^{n+k})\approx k\varPsi(t^{n+1})-(k-1)\varPsi(t^n).
	\end{align}
	We now establish a new consistent splitting scheme for the natural convection equations, which is second order. For all \( 1 \leq n \leq N \), we denote \( (\bm{u}^n, p^n, \theta^n) \) as the time-discrete approximations to \( (\bm{u}(t^n), p(t^n), \theta(t^n)) \). By decoupling the pressure term from the temperature and velocity terms, we can reduce the computational burden. More precisely, for any positive integer \( k \), the scheme proceeds as follows. Starting from the initial conditions \( (\bm{u}^0, \theta^0) = (\bm{u}_0, \theta_0) \) and \( (\bm{u}^1, \theta^1) \), the sequence \( \{\bm{u}^n, p^n, \theta^n\} \subset \bm{V} \times M \times \bm{X} \) is determined by the following second-order scheme, which satisfies the weak formulation.
	\begin{align}\label{3.5}
		& \frac{1}{2\delta t}\left( (2k+1)\bm{u}^{n+1} - 4k\bm{u}^{n} + (2k-1)\bm{u}^{n-1}, \bm{v} \right) + a\left(k\bm{u}^{n+1} - (k-1)\bm{u}^n, \bm{v}\right)\notag\\
		&  +b\big((k+1)\bm{u}^n - k\bm{u}^{n-1},{ (k+1)\bm{u}^n - k\bm{u}^{n-1}},\bm{v}\big) - d\left(\bm{v}, (k+1)p^n - kp^{n-1}\right)\notag\\
		&  = -Grg\nu^2\left( (k+1)\theta^n - k\theta^{n-1}, \bm{v} \right)+ \left( \bm{f}^{n+k}, \bm{v} \right),
	\end{align}
	\begin{align}\label{3.6}
		& \frac{1}{2\delta t}\left( (2k+1)\theta^{n+1} - 4k\theta^{n} + (2k-1)\theta^{n-1}, \bm{w} \right) + \lambda a\left(k\theta^{n+1} - (k-1)\theta^n, \bm{w}\right)\notag\\
		&   + b\left((k+1)\bm{u}^n - k\bm{u}^{n-1}, { (k+1)\theta^{n} - k\theta^{n-1}}, \bm{w}\right) = \left(b^{n+k}, \bm{w} \right),
	\end{align}
	\begin{align}\label{3.7}
		(\nabla\mathit{p}^{n+1},\nabla\mathit{q})=(\bm{f}^{n+1}-Grg\nu^2\theta^{n+1}-\nu\nabla\times\nabla\times\bm{u}^{n+1}-\bm{u}^{n+1}\cdot\nabla\bm{u}^{n+1},\nabla\mathit{q}).
	\end{align}
	for all $(\bm{v},q,\bm{w})\subset \bm{V}\times M\times\bm{X}$  with $1\leq n \leq N$. 
	We can easily check the scheme is second-order by $(3.1)-(3.2)$ and the fact that 
	\begin{align*}
		(k+1)p(t^{n+1})-kp(t^{n-1})=p(t^{n+k})-\frac{k(k+1)}{2}p^{''}(t^{n+k})\delta t^2+O(\delta t^3).
	\end{align*}
	
	Next, we provide supplementary details regarding $(\bm{u}^1, p^1, \theta^1)$. When $k = 1$, the present scheme reduces to the standard second-order consistent splitting scheme introduced in \cite{Gu2003}. Specifically, a first-order consistent splitting scheme for the natural convection equations based on projection methods was proposed in \cite{Z2016} we use the following scheme to compute $(\bm{u}^1, p^1, \theta^1)$.
	\begin{equation}\label{3.8}
		\left\{
		\begin{aligned}
			&\frac{\tilde{\bm{u}}^1 -\bm{u}^0}{\delta t} - \nu \Delta \tilde{\bm{u}}^1 + (\bm{u}^0 \cdot \nabla)\tilde{\bm{u}}^1+\nabla\Phi^0 = -k \nu^2 g \theta^1 + f(t^1), \\
			&\frac{\theta^1 - \theta^0}{\delta t} - \lambda \nu \Delta \theta^1 + (\bm{u}^0 \cdot \nabla) \theta^1 = b(t^1),\\
			&\left. \tilde{\bm{u}}^1 \right|_{\Gamma} = 0 ,
		\end{aligned}
		\right.
	\end{equation}
	and
	\begin{equation}\label{3.9}
		\left\{
		\begin{aligned}
			&\frac{\bm{u}^1-\tilde{\bm{u}}^1}{\delta t}+\alpha \nabla(\Phi^1-\Phi^0)=0,\\
			&\nabla \cdot \bm{u}^{1}=0,\\
			&\bm{u}^{1}\cdot\bm{n} |_{\Gamma} = 0.
		\end{aligned}
		\right.
	\end{equation}
	Where the $\alpha\geq1$ is a constant.
	For the stability and error estimates at \( n = 0 \), we can derive the following results using the method in\cite{Z2016}:
	\begin{align}
		&\|\nabla\bm{u}^1\|^2 + \|\nabla\theta^1\|^2  + \delta t\left(\|\Delta \bm{u}^1\|^2 + \|\Delta\theta^1\|^2\right) \leq C, \\
		&\|\bm{u}^1 - \bm{u}(t^1)\|^2 + \|\theta^1 - \theta(t^1)\|^2+\delta t(\|\nabla(\bm{u}^1 - \bm{u}(t^1))\|^2 + \|\nabla(\theta^1 - \theta(t^1))\|^2 + \|p^1 - p(t^1)\|^2) \leq C\delta t.
	\end{align}
	Here, \( C \) is a constant independent of \( \delta t \), more precise details can be found in \cite{Z2016}. Next, we will analyze the stability and error estimates for the new second-order consistent splitting scheme in the case of \( 1 \leq n \leq N \).
	
	\subsection{Stability Analysis}
	We perform a stability analysis for the new second-order scheme applied to the natural convection equations.
	\begin{theorem}\label{The3.1}
		{\it  For a fixed $T>0$, assume uniform bounds $\|b(\cdot,t)\|\leq C_b$ and $\|\bm{f}(\cdot,t)\|\leq C_f$, $\forall $$0\leqslant t\leqslant T$. Stability of the scheme $(\ref{3.5})-(\ref{3.7})$ holds when $k>4$ in both two and three dimensions, implying that for all $m\leq N$, 
			\begin{align}
				&\|\nabla\bm{u}^{m+1}\|^2+\|\nabla\theta^{m+1}\|^2+\delta t\sum_{n=0}^{m+1}(\|\Delta\theta^{n}\|^2+\|\Delta\bm{u}^n\|^2+\|\nabla p^{n}\|^2)\leq C_*  \quad    \forall m\leq n\leq\frac{T^*}{\delta t},
			\end{align}
			where the definition of \(T^*\) will be given in (\ref{3.34}).
		}
	\end{theorem}
	\begin{proof}
		We discuss the stability under two cases, because some Sobolev inequalities need to be discussed separately in two and three dimensions.
		\subsection*{Case 1: Rigorous proof of  stability for \(d = 2\).}
		First, we discuss the stability of $\bm{u}$ and $p$. In (\ref{3.5}), we set $\bm{v} = -\Delta\left((k + 1)\bm{u}^{n+1} - k \bm{u}^n\right)$. For the viscous term, we have
		\begin{align}\label{3.13}
			&a(k\bm{u}^{n+1}-(k-1)\bm{u}^n,-\Delta((k+1)\bm{u}^{n+1}-k\bm{u}^n))\nonumber\\
			= &\nu(\Delta(k\bm{u}^{n+1}-(k-1)\bm{u}^n),\Delta((k+1)\bm{u}^{n+1}-k\bm{u}^n))\nonumber\\
			= &\nu(\frac{k-1}{k}\Delta((k+1)\bm{u}^{n+1}-k\bm{u}^n)+\frac{1}{k}\Delta\bm{u}^{n+1},\Delta((k+1)\bm{u}^{n+1}-k\bm{u}^n))\nonumber\\
			= &\nu(\frac{k-1}{k}\|\Delta((k+1)\bm{u}^{n+1}-k\bm{u}^n)\|^2+\frac{1}{k}\|\Delta\bm{u}^{n+1}\|^2\nonumber\\
			&+\frac{1}{2}(\|\Delta\bm{u}^{n+1}\|^2-\|\Delta\bm{u}^n\|^2+\|\Delta\bm{u}^{n+1}-\Delta\bm{u}^n\|^2)).
		\end{align}
		For the nonlinear term, by making use of (\ref{2.3}) and Young’s inequality, we obtain the following estimate.
		{ 
		\begin{align}\label{3.14}
			&b((k+1)\bm{u}^n-k\bm{u}^{n-1},(k+1)\bm{u}^n-k\bm{u}^{n-1},-\Delta((k+1)\bm{u}^{n+1}-k\bm{u}^n))\nonumber\\
			\leq & c\|\nabla((k+1)\bm{u}^n-k\bm{u}^{n-1})\|^\frac{1}{2}\|(k+1)\bm{u}^n-k\bm{u}^{n-1}\|^{\frac{1}{2}}\|(k+1)\bm{u}^n-k\bm{u}^{n-1}\|^\frac{1}{2}_2\nonumber\\
			&\|\nabla((k+1)\bm{u}^n-k\bm{u}^{n-1})\|^\frac{1}{2}\|\Delta(k\bm{u}^{n+1}-(k-1)\bm{u}^n)\|\nonumber\\
			\leq & c(\varepsilon)\|\nabla((k+1)\bm{u}^n-k\bm{u}^{n-1})\|^3\|(k+1)\bm{u}^n-k\bm{u}^{n-1}\|_2 +\varepsilon\|\Delta(k\bm{u}^{n+1}-(k-1)\bm{u}^n)\|^2\nonumber\\
			\leq & c(\varepsilon) \|\nabla((k+1)\bm{u}^n-k\bm{u}^{n-1})\|^6+\varepsilon\|\Delta((k+1)\bm{u}^n-k\bm{u}^{n-1})\|^2 \nonumber\\
			&  +\varepsilon\|\Delta((k+1)\bm{u}^{n+1}-k\bm{u}^n)\|^2,
		\end{align}
	}where we have used the inequality \( \| \bm{u}\|_2^2 \leq C \|\Delta \bm{u}\|  \) and Theorem \ref{The3.1} in the last inequality.  For the pressure term, we deduce that
		\begin{align}\label{3.15}
			d(-\Delta((k+1)\bm{u}^{n+1}-k\bm{u}^n),(k+1)p^n-kp^{n-1})
			\leq  \|\nabla((k+1)p^n-kp^{n-1})\|\|\Delta((k+1)\bm{u}^{n+1}-k\bm{u}^n)\|.
		\end{align} 
		We can obtain from \cite{J2007} that
		\begin{align}
			(\nabla p_s(\bm{u}),\nabla q)=-(\nabla\times\nabla\times\bm{u},\nabla q).
		\end{align}
		Then by (\ref{3.7}), we can get that
		\begin{align}
			&(\nabla((k+1)p^n-kp^{n-1}),\nabla q)\nonumber\\
			=&((1+k)\bm{f}^n-k\bm{f}^{n-1}-Gr\nu^2g((k+1)\theta^n-k\theta^{n-1}),\nabla q)\nonumber\\
			&-({ (k+1)\bm{u}^n\cdot\nabla\bm{u}^n+k\bm{u}^{n-1}\cdot\nabla\bm{u}^{n-1}}+\nabla p_s((k+1)\bm{u}^n-k\bm{u}^{n-1}),\nabla q).
		\end{align}
		Now, taking $q=((k+1)p^n-kp^{n-1})$, we obtain that
		\begin{align}\label{3.18}
			&\|\nabla((k+1)p^n-kp^{n-1})\|\nonumber\\
			&\leq \|(1+k)\bm{f}^n-k\bm{f}^{n-1}-Gr\nu^2g((k+1)\theta^n+k\theta^{n-1})-{ (k+1)\bm{u}^{n}\cdot\nabla\bm{u}^n+k\bm{u}^{n-1}\cdot\nabla\bm{u}^{n-1}}\|\nonumber\\
			& \quad+\|\nabla p_s((k+1)\bm{u}^n-k\bm{u}^{n-1})\|.
		\end{align}
		Then, by using the Sobolev  embedding inequality , we have
		\begin{align}\label{3.19}
			&\|(1+k)\bm{f}^n-k\bm{f}^{n-1}-Gr\nu^2g((k+1)\theta^n-k\theta^{n-1})-{ (k+1)\bm{u}^n\cdot\nabla\bm{u}^n+k\bm{u}^{n-1}\cdot\nabla\bm{u}^{n-1}}\|^2 \nonumber\\
			\leq & N(\|(1+k)\bm{f}^n-k\bm{f}^{n-1}\|^2+\|Gr\nu^2g((k+1)\theta^n-k\theta^{n-1})\|^2\nonumber\\
			&+{ (k+1)^2\|\bm{u}^n\cdot\nabla\bm{u}^n\|^2+k^2\|\bm{u}^{n-1}\cdot\nabla\bm{u}^{n-1}\|^2})\nonumber\\
			\leq &  N(\|(1+k)\bm{f}^n-k\bm{f}^{n-1}\|^2+\|Gr\nu^2g((k+1)\theta^n-\theta^{n-1})\|^2+C{ \|\bm{u}^n\|\|\nabla\bm{u}^n\|^2\|\Delta\bm{u}^n\|}\nonumber\\
			&  +{ C\|\bm{u}^{n-1}\|\|\nabla\bm{u}^{n-1}\|^2\|\Delta\bm{u}^{n-1}\|}\nonumber\\
			\leq & N(\|(1+k)\bm{f}^n-k\bm{f}^{n-1}\|^2+\|Gr\nu^2g((k+1)\theta^n-k\theta^{n-1})\|^2)+\varepsilon{ (\|\Delta\bm{u}^n\|^2+\|\Delta\bm{u}^{n-1}\|^2)}\nonumber\\
			&+c(\varepsilon){ (\|\nabla\bm{u}^{n}\|^6+\|\nabla \bm{u}^{n-1}\|^6)},
		\end{align}
		where we have used the following inequality \cite{J2007}
		\begin{align*}
			&\|\bm{u}^n\cdot\nabla\bm{u}^n\|^2\leq\|\bm{u}^n\|_{L_4}^2\|\nabla\bm{u}^n\|_{L_4}^2\leq\|\bm{u}^n\|\|\nabla\bm{u}^n\|^2\|\Delta\bm{u}^n\|,  \quad    d=2.
		\end{align*}
		As a consequence, we can estimate the pressure term by making use of Lemma\ref{Lem2.2} as
		\begin{align}
			&d(-\Delta((k+1)\bm{u}^{n+1}-k\bm{u}^n),(k+1)p^n-kp^{n-1})\nonumber\\
			\leq &\|\Delta((k+1)\bm{u}^{n+1}-k\bm{u}^n)\|(\|(1+k)\bm{f}^n-k\bm{f}^{n-1}-Gr\nu^2g((k+1)\theta^n-k\theta^{n-1})\nonumber\\
			&  { -(k+1)\bm{u}^n\cdot\nabla\bm{u}^n-k\bm{u}^{n-1}\cdot\nabla\bm{u}^{n-1}\|}+\|\nabla p_s((k+1)\bm{u}^n-k\bm{u}^{n-1})\|)\nonumber\\
			\leq &(\|(1+k)\bm{f}^n-k\bm{f}^{n-1}\|+\|Gr\nu^2g((k+1)\theta^n-k\theta^{n-1})\|+{ \|(k+1)\bm{u}^n\cdot\nabla\bm{u}^n\|}\nonumber\\
			&  +\|k\bm{u}^{n-1}\cdot\nabla\bm{u}^{n-1}\|)\|\Delta((k+1)\bm{u}^{n+1}-k\bm{u}^n)\|+\frac{1}{2}\|\nabla p_s\|^2+\frac{1}{2}\|\Delta((k+1)\bm{u}^{n+1}-k\bm{u}^n)\|^2\nonumber\\
			\leq & Nc(\varepsilon)(\|(1+k)\bm{f}^n-k\bm{f}^{n-1}\|^2+\|Gr\nu^2g((k+1)\theta^n-k\theta^{n-1})\|^2)+c(\varepsilon)({ \|\nabla\bm{u}^{n}\|^6+\|\nabla \bm{u}^{n-1}\|^6})\nonumber\\
			&  +\varepsilon(\|\Delta\bm{u}^n\|^2+\|\Delta\bm{u}^{n-1}\|^2+\|\Delta((k+1)\bm{u}^{n+1}-k\bm{u}^n)\|^2)+(\frac{1}{4}+\frac{\varepsilon}{2})\|\Delta((k+1)\bm{u}^{n}-k\bm{u}^{n-1})\|^2\nonumber\\
			&  +C\|\nabla((k+1)\bm{u}^{n}-k\bm{u}^{n-1})\|^2+\frac{1}{2}\|\Delta((k+1)\bm{u}^{n+1}-k\bm{u}^n)\|^2.
		\end{align}
		For the body force and heat sorces term, we have
		\begin{align}\label{3.22}
			(\bm{f}^{n+k},-\Delta((k+1)\bm{u}^{n+1}-k\bm{u}^n)) \leq c(\varepsilon)\|\bm{f}^{n+k}\|^2+\varepsilon\|\Delta((k+1)\bm{u}^{n+1}-k\bm{u}^n)\|^2. 
		\end{align}
		\begin{align}\label{3.23}
			&(Gr\nu^2g((k+1)\theta^n-k\theta^{n-1}),-\Delta((k+1)\bm{u}^{n+1}-k\bm{u}^n))\nonumber\\
			&\leq c(\varepsilon)\|Gr\nu^2g((k+1)\theta^n-k\theta^{n-1})\|^2+\varepsilon\|\Delta((k+1)\bm{u}^{n+1}-k\bm{u}^n)\|^2.
		\end{align}
		To ensure stability, we conclude from (\ref{3.13})-(\ref{3.23}) that the parameter \( k \) satisfies the following constraints:
		\begin{align}
			&\frac{k-1}{k}-4\varepsilon-\frac{1}{2}-\varepsilon>\frac{1}{4}+\frac{\varepsilon}{2}  \quad \text{ and }  \quad \frac{1}{k}>4\varepsilon.\nonumber
		\end{align}
		This implies that \( k > 4 \). Furthermore, $k$ must be a positive integer, as the schemes (\ref{3.5})-(\ref{3.7}) require the use of \(\bm{f}^{n+k}\) and \(b^{n+k}\).
		
		Now, we proceed to deal with the time derivative term. To simplify the proof process, we fix \( k = 5 \), \( \nu = 1 \), and \( \lambda = 1 \). By the method of undetermined coefficients (as in (\ref{3.13})), we derive that
		\begin{align}\label{3.24}
			&(11\bm{u}^{n+1} - 20\bm{u}^n + 9\bm{u}^{n-1}, -\Delta(6\bm{u}^{n+1} - 5\bm{u}^n)) \nonumber\\
			\nonumber
			= & \frac{1}{10} \left( \|\nabla \bm{u}^{n+1}\|^2 - \|\nabla \bm{u}^n\|^2 \right) + \left\| \frac{9\sqrt{10}}{5} \nabla \bm{u}^{n+1} - \frac{\sqrt{90}}{2} \nabla \bm{u}^n \right\|^2 - \left\| \frac{9\sqrt{10}}{5} \nabla \bm{u}^n - \frac{\sqrt{90}}{2} \nabla \bm{u}^{n-1} \right\|^2 \\
			\nonumber
			&  + \left\| \frac{\sqrt{90}}{2} \nabla \bm{u}^{n+1} - \sqrt{90} \nabla \bm{u}^n + \frac{\sqrt{90}}{2} \nabla \bm{u}^{n-1} \right\|^2  + \frac{13}{2} \|\nabla(\bm{u}^{n+1} - \bm{u}^n)\|^2 - \frac{9}{2} \|\nabla(\bm{u}^n - \bm{u}^{n-1})\|^2\nonumber\\
			& + \frac{9}{2} \|\nabla(\bm{u}^{n+1} - 2\bm{u}^n + \bm{u}^{n-1})\|^2.
		\end{align}
		Summing up (\ref{3.13})–(\ref{3.24}) and  neglecting non-essential terms, we obtain that
		\begin{align}\label{3.25}
			& \frac{1}{10} \left( \|\nabla \bm{u}^{n+1}\|^2 - \|\nabla \bm{u}^n\|^2 \right) + \left\| \frac{9\sqrt{10}}{5} \nabla \bm{u}^{n+1} - \frac{\sqrt{90}}{2} \nabla \bm{u}^n \right\|^2+\delta t(\|\nabla\bm{u}^{n+1}\|^2-\|\nabla\bm{u}^{n}\|^2) \nonumber\\
			&   - \left\| \frac{9\sqrt{10}}{5} \nabla \bm{u}^n - \frac{\sqrt{90}}{2} \nabla \bm{u}^{n-1} \right\|^2 + \left\| \frac{\sqrt{90}}{2} \nabla \bm{u}^{n+1} - \sqrt{90} \nabla \bm{u}^n + \frac{\sqrt{90}}{2} \nabla \bm{u}^{n-1} \right\|^2 \nonumber\\
			&   + \frac{13}{2} \|\nabla(\bm{u}^{n+1} - \bm{u}^n)\|^2 - \frac{9}{2} \|\nabla(\bm{u}^n - \bm{u}^{n-1})\|^2 +\frac{8\delta t}{5}\|6\bm{u}^{n+1}-5\bm{u}^n\|^2+\frac{2\delta t}{5}\|\Delta\bm{u}^{n+1}\|^2\nonumber\\
			\leq & c(\varepsilon)\delta t(\|\nabla\bm{u}^n\|^6+\|\nabla\bm{u}^{n-1}\|^6)+(\frac{1}{2}+\varepsilon)\delta t\|\Delta(6\bm{u}^{n}-5\bm{u}^{n-1})\|^2\nonumber\\
			&  +2\varepsilon\delta t(\|\Delta(6\bm{u}^{n}-5\bm{u}^{n-1})\|^2+\|\Delta\bm{u}^n\|^2+\|\Delta\bm{u}^{n-1}\|^2)+c(\varepsilon) \delta t{ \|\nabla(6\bm{u}^n-5\bm{u}^{n-1})\|^2}\nonumber\\
			&  +c(\varepsilon) \delta t\|\nabla(6\bm{u}^n-5\bm{u}^{n-1})\|^6+c(\varepsilon)\delta t(\|\bm{f}^{n+5}\|^2+\|6\bm{f}^n-5\bm{f}^{n-1}\|^2+Gr^2g^2\|6\theta^n-5\theta^{n-1}\|^2).
		\end{align}
		Now, by choosing \( \varepsilon = \frac{1}{100} \) and take the summation for $n=1$ up to $m$ (where $m\leq\frac{T}{\delta t}-1$) in (\ref{3.25}), we obtain
		{ 
		\begin{align}\label{3.26}
			&\left\| \nabla \bm{u}^{m+1} \right\|^2 + \delta t \sum_{n=0}^m \left\| \Delta \bm{u}^n \right\|^2\nonumber\\
			&\leq C \delta t \sum_{n=0}^{m} \Big( \|\bm{f}^{n+5}\|^2 + \big\|6\bm{f}^n - 5\bm{f}^{n-1}\big\|^2 + Gr^2 g^2 \big\|6\theta^n - 5\theta^{n-1}\big\|^2+\left\| \nabla \bm{u}^n \right\|^6 +\left\| \nabla \bm{u}^n \right\|^2\Big) \nonumber \\
			&\leq C \delta t \sum_{n=0}^{m} \Big(\left\| \nabla \bm{u}^n \right\|^6+\left\| \nabla \bm{u}^n \right\|^2 +\left\| \nabla \theta^n \right\|^2\Big)
			+ C T C_f^2, \quad \forall\, 0 \le m \leq N-1.
		\end{align}
		}
		Next, we consider the stability of $\theta$, let$\bm{w}=-\Delta((k+1)\theta^{n+1}-k\theta^n)$ in (\ref{3.6}). As same as  (\ref{3.13}), there holds that
		\begin{align}\label{3.27}
			& a(k\theta^{n+1}-(k-1)\theta^n,-\Delta((k+1)\theta^{n+1}-k\theta^n))\\
			= &\frac{k-1}{k}\|\Delta((k+1)\theta^{n+1}-k\theta^n)\|^2+\frac{1}{k}\|\Delta\theta^{n+1}\|^2+\frac{1}{2}(\|\Delta\theta^{n+1}\|^2-\|\Delta\theta^n\|^2+\|\Delta\theta^{n+1}-\Delta\theta^n\|^2).\nonumber
		\end{align}
		For the nonlinear term, as same as  (\ref{3.14}) we get that
		{ 
		\begin{align}\label{3.28}
			&b((1+k)\bm{u}^n-k\bm{u}^{n-1},(k+1)\theta^{n} - k\theta^{n-1},-\Delta((k+1)\theta^{n+1}-k\theta^n))\nonumber\\
			\leq & c\|\Delta((k+1)\theta^{n+1}-k\theta^n)\|\|\nabla((1+k)\bm{u}^n-k\bm{u}^{n-1})\|^{\frac{1}{2}}\|(1+k)\bm{u}^n-k\bm{u}^{n-1}\|^\frac{1}{2}\nonumber\\
			&  \|\nabla((k+1)\theta^{n} - k\theta^{n-1})\|^\frac{1}{2}\|(k+1)\theta^{n} - k\theta^{n-1}\|^\frac{1}{2}_2\nonumber\\
			\leq & c(\varepsilon)\|\nabla((1+k)\bm{u}^n-k\bm{u}^{n-1})\|^2\|\Delta((k+1)\theta^{n} - k\theta^{n-1}))\|\|\nabla((k+1)\theta^{n} - k\theta^{n-1})\|\nonumber\\
			&  +\varepsilon\|\Delta((k+1)\theta^{n+1}-k\theta^n)\|^2\nonumber\\
			\leq & c(\varepsilon)\|\nabla((1+k)\bm{u}^n-k\bm{u}^{n-1})\|^8+\varepsilon\|\nabla((k+1)\theta^{n} - k\theta^{n-1})\|^4\nonumber\\
			&  +\varepsilon\|\Delta((k+1)\theta^{n+1}-k\theta^n)\|^2+\varepsilon\|\Delta((k+1)\theta^{n} - k\theta^{n-1})\|^2.
		\end{align}}
		For the right-hand term
		\begin{align}\label{3.29}
			&(b^{n+k},-\Delta((k+1)\theta^{n+1}-k\theta^n))\leq c(\varepsilon)\|b^{n+k}\|^2+\varepsilon\|\Delta((k+1)\theta^{n+1}-k\theta^n)\|^2.
		\end{align}
		For consistency with the preceding proof, we choose \(k = 5\) as a specific example, since it satisfies the stability condition \(\frac{k-1}{k} > 3\varepsilon\)  from (\ref{3.27})-(\ref{3.29}).  Following the same method of  (\ref{3.13}), we get that
		\begin{align}\label{3.30}
			&(11\theta^{n+1} - 20\theta^n + 9\theta^{n-1}, -\Delta(6\theta^{n+1} - 5\theta^n)) \nonumber\\
			=& \frac{1}{10} \left( \|\nabla \theta^{n+1}\|^2 - \|\nabla \theta^n\|^2 \right) + \left\| \frac{9\sqrt{10}}{5} \nabla \theta^{n+1} - \frac{\sqrt{90}}{2} \nabla \theta^n \right\|^2 - \left\| \frac{9\sqrt{10}}{5} \nabla \theta^n - \frac{\sqrt{90}}{2} \nabla \theta^{n-1} \right\|^2 	\nonumber\\
			&  + \left\| \frac{\sqrt{90}}{2} \nabla \theta^{n+1} - \sqrt{90} \nabla \theta^n + \frac{\sqrt{90}}{2} \nabla \theta^{n-1} \right\|^2 + \frac{13}{2} \|\nabla(\theta^{n+1} - \theta^n)\|^2 - \frac{9}{2} \|\nabla(\theta^n - \theta^{n-1})\|^2\nonumber\\	
			&   + \frac{9}{2} \|\nabla(\theta^{n+1} - 2\theta^n + \theta^{n-1})\|^2.
		\end{align}
		Summing up (\ref{3.27})-(\ref{3.30}), after dropping some unnecessary terms we obtain
		\begin{align}\label{3.31}
			&\frac{1}{10}(\|\nabla\theta^{n+1}\|^2-\|\nabla\theta^{n}\|^2)+\|\frac{9\sqrt{10}}{5}\nabla\theta^{n+1}-\frac{\sqrt{90}}{2}\nabla\theta^{n}\|^2-\|\frac{9\sqrt{10}}{5}\nabla\theta^{n}-\frac{\sqrt{90}}{2}\nabla\theta^{n-1}\|^2\nonumber\\
			&  +\frac{13}{2}\|\nabla(\theta^{n+1}-\theta^n)\|^2-\frac{9}{2}\|\nabla(\theta^{n}-\theta^{n-1})\|^2+\frac{8}{5}\delta t\|\Delta(6\theta^{n+1}-5\theta^n)\|^2+\frac{2}{5}\delta t\|\Delta\theta^{n+1}\|^2\nonumber\\
			&  +\delta t(\|\Delta\theta^{n+1}\|^2-\|\Delta\theta^{n}\|^2+\|\Delta\theta^{n+1}-\Delta\theta^{n}\|^2)\nonumber\\
			\leq & { 2 c(\varepsilon) \delta t\|\nabla(6\bm{u}^{n}-5\bm{u}^{n-1})\|^8}+2\delta t \varepsilon\|\nabla(6\theta^{n}-5\theta^{n-1})\|^4+c(\varepsilon)\delta t\|b^{n+k}\|^2 \nonumber\\
			&\quad+2\delta t \varepsilon(\|\Delta(6\theta^{n+1}-5\theta^n)\|^2+\|\Delta(6\theta^{n} - 5\theta^{n-1})\|^2).
		\end{align}
		Taking the sum over \( n \) from \( 1 \) to \( m \) in (\ref{3.31}) and making \( \varepsilon = \frac{1}{100} \), we obtain
		{ 
		\begin{align}\label{3.32}
			& \|\nabla\theta^{m+1}\|^2+\delta t\sum_{n=0}^{m}\|\Delta\theta^{n+1}\|^2\nonumber\\
			 \leq& 2C\delta t \sum_{n=0}^{m}\big\|\nabla\big(6\bm{u}^{n}-5\bm{u}^{n-1}\big)\big\|^8
			+ 2C\delta t \sum_{n=0}^{m}(\big\|\nabla\big(6\theta^{n}-5\theta^{n-1}\big)\big\|^4
			+ c(\varepsilon)\|b^{n+k}\|^2)\nonumber \\
			 \leq& C\delta t\sum_{n=0}^{m}\big(\|\nabla\bm{u}^n\|^8+\|\nabla\theta^n\|^4\big) + CTC_b^2
		\end{align}}
		Combining (\ref{3.26})and(\ref{3.32}), we get that
		By using Lemma \ref{Lem2.3} we can get that
		{ 
		\begin{align}\label{3.33}
			&\|\nabla\bm{u}^{m+1}\|^2+\|\nabla\theta^{m+1}\|^2+\delta t\sum_{n=0}^{m}(\|\Delta\theta^{n}\|^2+\|\Delta\bm{u}^n\|^2)\nonumber\\
			\leq & C\sum_{n=0}^{m}((\|\nabla\bm{u}^n\|^2+\|\nabla\theta^{n}\|^2)^4+(\|\nabla\bm{u}^n\|^2+\|\nabla\theta^{n}\|^2)^3+(\|\nabla\theta^n\|^2+\|\nabla\bm{u}^n\|^2)^2\nonumber\\
			& +(\|\nabla\theta^n\|^2+\|\nabla\bm{u}^n\|^2))+CT(C_f^2+C_b^2).
		\end{align}}
	{ 
	We define \(\Psi = x^4 + x^3 + x^2+x \), where \(\Psi: (0, \infty) \to (0, \infty) \) is being a continuous and increasing function, and such that}
		\begin{align}\label{3.34}
			0 < T^* < \int_{C M^2 + C T C_f^2 + C T C_{\theta}^2}^{\infty} \frac{dz}{\Psi(z)},
		\end{align}
		then Lemma \ref{Lem2.5} implies that there exists $C_*>0$ independent of $\delta t>0$, such that
		{ 
		\begin{align}\label{3.35}
			&\|\nabla\bm{u}^{m+1}\|^2+\|\nabla\theta^{m+1}\|^2+\delta t\sum_{n=1}^{m+1}(\|\Delta\theta^{n}\|^2+\|\Delta\bm{u}^n\|^2)\leq C_*  \quad    \forall m\leq n\leq\frac{T^*}{\delta t}.
		\end{align}}
	For the bound of the pressure we can get from (\ref{3.18}) and $(\ref{3.19})$, thus, we complete the proof of this Theorem \ref{The3.1} when $d=2$. 
		\subsection*{Case 2: Rigorous proof of  stability for \(d = 3\).}
		For $d=3$ the proof proceeds in substantially the same manner as the proof for $d=2$. Owing to the invalidity of (\ref{2.3}), we are required to address (\ref{3.14}), (\ref{3.19}) and (\ref{3.28}). For simplicity, we only point out the handling method for (\ref{3.14}) and (\ref{3.19}). For the nonlinear term, by using (\ref{2.4}), we have that
		{ 
			\begin{align}
			&(((k+1)\bm{u}^n-k\bm{u}^{n-1})\cdot\nabla((k+1)\bm{u}^n-k\bm{u}^{n-1}),-\Delta((k+1)\bm{u}^{n+1}-k\bm{u}^n))\nonumber\\
			\leq & C\|(k+1)\bm{u}^n-k\bm{u}^{n-1}\|_1\|(k+1)\bm{u}^n-k\bm{u}^{n-1}\|_1^\frac{1}{2}\|(k+1)\bm{u}^n-k\bm{u}^{n-1}\|_2^\frac{1}{2}\|\Delta((k+1)\bm{u}^{n+1}-k\bm{u}^n)\|\nonumber\\
			\leq & C(\varepsilon)\|\nabla((k+1)\bm{u}^n-k\bm{u}^{n-1})\|^2\|\nabla((k+1)\bm{u}^n-k\bm{u}^{n-1})\|\|\Delta((k+1)\bm{u}^n-k\bm{u}^{n-1})\|\nonumber\\
			&  +\varepsilon\|\Delta((k+1)\bm{u}^{n+1}-k\bm{u}^n)\|^2\nonumber\\
			\leq &
			C(\varepsilon)(\|\nabla((k+1)\bm{u}^n-k\bm{u}^{n-1})\|^3\|\Delta((k+1)\bm{u}^n-k\bm{u}^{n-1})\|) +\varepsilon\|\Delta((k+1)\bm{u}^{n+1}-k\bm{u}^n)\|^2\nonumber\\
			\leq & C(\varepsilon)\|\nabla((k+1)\bm{u}^n-k\bm{u}^{n-1})\|^6
			+\varepsilon\|\Delta((k+1)\bm{u}^n-k\bm{u}^{n-1})\|^2+\varepsilon\|\Delta((k+1)\bm{u}^{n+1}-k\bm{u}^n)\|^2.
			\end{align}}
		For the pressure term, we can deduce that
		\begin{align}
			&\|(1+k)\bm{f}^n-k\bm{f}^{n-1}-Grg((k+1)\theta^n-k\theta^{n-1})-{ (k+1)\bm{u}^n\cdot\nabla\bm{u}^n-k\bm{u}^{n-1}\cdot\nabla\bm{u}^{n-1}}\|^2\nonumber\\
			\leq & N(\|(1+k)\bm{f}^n-k\bm{f}^{n-1}\|^2+\|Grg((k+1)\theta^n+k\theta^{n-1})\|^2\nonumber\\
			& +{ (k+1)^2\|\bm{u}^n\cdot\nabla\bm{u}^n\|^2+k^2\|\bm{u}^{n-1}\cdot\nabla\bm{u}^{n-1}\|^2})\nonumber\\
			\leq & N(\|(1+k)\bm{f}^n-k\bm{f}^{n-1}\|^2+\|Grg((k+1)\theta^n-k\theta^{n-1})\|^2)\nonumber\\
			&+C{ \|\nabla\bm{u}^n\|^3\|\Delta\bm{u}^n\|+C\|\nabla\bm{u}^{n-1}\|^3\|\Delta\bm{u}^{n-1}\|}\nonumber\\
			\leq & N(\|(1+k)\bm{f}^n-k\bm{f}^{n-1}\|^2+\|Grg((k+1)\theta^n-k\theta^{n-1})\|^2)  +C(\varepsilon)({ \|\nabla\bm{u}^n\|^6+\|\nabla\bm{u}^{n-1}\|^6})\nonumber\\
			&+\varepsilon\|\Delta\bm{u}^n\|^2+\varepsilon\|\Delta\bm{u}^{n-1}\|^2,
		\end{align}
		where we have used the following inequalities
		\begin{align}
			&\|\bm{u}^n\cdot\nabla\bm{u}^n\|^2\leq\|\bm{u}^n\|_{L_6}^2\|\nabla\bm{u}^n\|_{L_3}^2\leq C\|\nabla\bm{u}^n\|^3\|\Delta\bm{u}^n\|.
		\end{align}
		The rest of the estimates are the same as in the two-dimensional case.  To ensure stability, we impose the following conditions on the parameter \( k \),  
		\begin{align*}
			&\frac{1-k}{k}>4\varepsilon+\frac{1}{4}+\frac{\varepsilon}{2}+\frac{1}{2}  \quad\text{and} \quad \frac{1}{k}>4\varepsilon,
		\end{align*}
		which implies $k>4$. Now, we also fix $k=5$ , $\varepsilon=\frac{1}{50}$ combining other estimates, we get that
		{ \begin{align}
			&\|\nabla\bm{u}^{m+1}\|^2 + \delta t \sum_{n=0}^{m+1} \|\Delta\bm{u}^n\|^2 \label{eq:u-estimate} \nonumber\\
			\leq & C \sum_{n=0}^{m+1} \|\nabla\bm{u}^{n}\|^8 + C \sum_{n=0}^{m+1} \|\nabla\bm{u}^{n}\|^6 + C \sum_{n=0}^{m+1} \|\nabla\bm{u}^{n}\|^4 + C M^2 + T C C_f^2 + T C C_{\theta}^2.
		\end{align}}
	As similar method of (\ref{3.33})-(\ref{3.35})
		The stability estimates for \(\theta\) and $p$ follow the same line of reasoning as those for the velocity \(\bm{u}\). With this, the proof of Theorem \ref{The3.1} is completed, when $d=3$. Notably, the stability results established for the three-dimensional scenario are restricted to a local time \(T^*\), which satisfies the condition given in (\ref{3.34}). Having addressed all cases both two-dimensional and three-dimensional, the proof of Theorem 3.2 is thus fully completed.  
	\end{proof}
	\section{Error Estimates}
	\setcounter{equation}{0}
	The error estimate associated with the new second-order consistent splitting scheme is presented in this section, where the scheme is designed for solving the natural convection equations.
	\begin{theorem}\label{The4.1}
		{Let \( d = 2, 3 \) and \( T > 0 \). Let \( \bm{u}(t^n) \), \( p(t^n) \) and \( \theta(t^n) \) be the solutions of (\ref{1.1}), and suppose that \( \bm{u}^n \), \( p^n \), and \( \theta^n \) are computed via the second-order scheme (\ref{3.5})–(\ref{3.7}). To simplify the analysis, we set \( k = 5 \), \( \nu = 1 \), and \( \lambda = 1 \). We need the following regularities of the exact solutions:
			\begin{align}
				\frac{\partial^2\bm{u}}{\partial t^3}\in\mathit{L}^2(0,\mathit{T};\mathit{H^2}), \frac{\partial^3\bm{u}}{\partial t^2}\in\mathit{L}^2(0,\mathit{T};\mathit{L^2}), 
				\frac{\partial^2\theta}{\partial t^2}\in\mathit{L}^2(0,\mathit{T};\mathit{H^2}), \frac{\partial^3\theta}{\partial t^3}\in\mathit{L}^2(0,\mathit{T};\mathit{L^2}), \frac{\partial^2p}{\partial t^2}\in\mathit{L}^2(0,\mathit{T};\mathit{H^1}).\nonumber
			\end{align}
			Then for $m\leq N$, we have
			\begin{align}
				\|\nabla\bm{e}_u^{m+1}\|^2+\|\nabla\eta^{m+1}\|^2+\delta t\sum_{n=0}^{m+1}(\|\Delta\bm{e}_u^n\|^2+ \|\Delta\eta^n\|^2+\|\nabla e_p^n\|^2)\leq C\delta t^4,\nonumber
			\end{align}
			where the constant $C$ is  independent of $\delta t$, but is dependent on  $T$, $\Omega$ and the exact solutions of $\bm{u}$ and $\theta$.
		}
	\end{theorem}
	\begin{proof}
		Before discussing the error of the scheme, we first derive the error equations. To simplify the analysis, we set \( k = 5 \), \( \nu = 1 \), and \( \lambda = 1 \), and the scheme can thus be rewritten in the following form.
		\begin{equation}\label{4.1}
			\left\{
			\begin{aligned}
				&\frac{1}{2\delta t}\left( 11\bm{u}^{n+1} - 20\bm{u}^{n} + 9\bm{u}^{n-1}, \bm{v} \right) + a\left(5\bm{u}^{n+1} - 4\bm{u}^n, \bm{v}\right) + b\left(6\bm{u}^n - 5\bm{u}^{n-1},{  6\bm{u}^n - 5\bm{u}^{n-1}}, \bm{v}\right)\\
				& - d\left(\bm{v}, 6p^n - 5p^{n-1}\right)= -Gr g \left( 6\theta^n - 5\theta^{n-1}, \bm{v} \right) + \left( \bm{f}^{n+5}, \bm{v} \right), \\
				&\frac{1}{2\delta t}\left( 11\theta^{n+1} - 20\theta^{n-1} + 9\theta^n, \bm{w} \right) +  a\left(5\theta^{n+1} - 4\theta^n, \bm{w}\right)+ b\left(6\bm{u}^n - 5\bm{u}^{n-1}, { 6\theta^{n} - 5\theta^{n-1}}, \bm{w}\right)=\left(b^{n+5}, \bm{w} \right), \\
				&\left(\nabla p^{n+1}, \nabla q\right) = \left( \bm{f}^{n+1} - Gr g  \theta^{n+1} - \bm{u}^{n+1} \cdot \nabla \bm{u}^{n+1} - \nu \nabla \times \nabla \times \bm{u}^{n+1}, \nabla q \right).
			\end{aligned}
			\right.
		\end{equation}
		We denote the errors as follows: \( \bm{e}_{\bm{u}}^n = \bm{u}^n - \bm{u}(\cdot, t^n) \) (velocity error), \( \eta = \theta^n - \theta(\cdot, t^n) \) (temperature error), and \( e_p^n = p^n - p(\cdot, t^n) \) (pressure error).  
		From (\ref{4.1}) and (\ref{1.1}), we can derive the error equations for \( \bm{u}^{n+1} \):   
		\begin{align}\label{4.2}
			&(11\bm{e}_u^{n+1} - 20\bm{e}_u^{n} + 9\bm{e}_u^{n-1}, \bm{v}) 
			+ 2\delta t\left[ a(5\bm{e}_u^{n+1} - 4\bm{e}_u^{n}, \bm{v}) + b(6\bm{u}^{n} - 5\bm{u}^{n-1}, { 6\bm{u}^{n} - 5\bm{u}^{n-1}}, \bm{v}) \right. \nonumber \\
			&  - b(6\bm{u}(t^{n}) - 5\bm{u}(t^{n-1}), { 6\bm{u}(t^{n})-5\bm{u}(t^{n-1})},\bm{v})+ Grg(6\eta^n - 5\eta^{n-1}, \bm{v}) + d(\bm{v}, 6e_p^n - 5e_p^{n-1})\left. \right] \nonumber \\
			= &( R^n + 2\delta t(Q^n + S^n + H^n + B^n),\bm{v}),
		\end{align}
		where $R^n,Q^n,S^n,H^n,B^n$ are given by
		\begin{align}
			&R^n=-11\bm{u}(t^{n+1})+20\bm{u}(t^{n})-9\bm{u}(t^{n-1})+2\delta t\bm{u}_t(t^{n+5}),\\
			&Q^n=-\Delta\bm{u}(t^{n+5})+\Delta(5\bm{u}(t^{n+1})-4\bm{u}(t^{n})),\\
			&{ S^n=\bm{u}(t^{n+5})\cdot\nabla\bm{u}(t^{n+5})-(6\bm{u}(t^{n})-5\bm{u}(t^{n-1}))\cdot\nabla(6\bm{u}(t^{n})-5\bm{u}(t^{n-1}))},\label{4.5}\\
			&H^n=\theta(t^{n+5})-(6\theta(t^{n})-5\theta(t^{n-1})),\\
			&B^n=\nabla p(t^{n+5})-\nabla(6p(t^{n})-5p(t^{n-1})).
		\end{align}
		Now, we can get the error equation for $\theta^{n+1}$ as follows 
		\begin{align}
			&\left(11\eta^{n+1} - 20\eta^{n} + 9\eta^{n-1}, \bm{w}\right) 
			+ 2\delta t \left[ a\left(5\eta^{n+1} - 4\eta^{n}, \bm{w}\right) \right. + b\left(6\bm{u}^n - 5\bm{u}^{n-1}, 6\theta^{n} - 5\theta^{n-1}, \bm{w}\right)\nonumber\\
			&  \left. - b\left(6\bm{u}(t^n) - 5\bm{u}(t^{n-1}), {  6\theta(t^{n}) - 5\theta(t^{n-1})}, \bm{w}\right) \right] = (T^n + 2\delta t \left(W^n + E^n\right),\bm{w}), \label{4.8}
		\end{align}
		where, $T^n,W^n,E^n$ are given by
		
		\begin{align}
			&T^n=-11\theta(t^{n+1})+20\theta(t^{n})-9\theta(t^{n-1})+2\delta t\theta_t(t^{n+5}),\\
			&W^n=-\Delta\theta(t^{n+5})+\Delta(5\theta(t^{n+1})-4\theta(t^{n})),\label{4.10}\\
			&{ E^n=\bm{u}(t^{n+5})\cdot\nabla\theta(t^{n+5})-(6\bm{u}(t^{n})-5\bm{u}(t^{n-1}))\cdot\nabla(6\theta(t^{n})-5\theta(t^{n-1})).}
		\end{align}
		By using the Stokes pressure and (\ref{4.1}), we can obtain the pressure term error equation as:
		\begin{align}\label{4.12}
			&(\nabla e_p^n,\nabla q)=(-Grg\eta^n+\bm{u}(t^{n})\cdot\nabla\bm{u}(t^{n})-\bm{u}^{n}\cdot\nabla\bm{u}^{n},\nabla q)+(\nabla p_s(\bm{e}_u^n),\nabla q).
		\end{align}
		Furthermore, we have that
		\begin{align}\label{4.13}
			& (\nabla (6\bm{e}_p^n-5\bm{e}_p^{n-1}),\nabla q)\\=&(6\bm{u}(t^n)\cdot\nabla\bm{u}(t^n)-6\bm{u}^n\cdot\nabla\bm{u}^n-5\bm{u}(t^{n-1})\cdot\nabla\bm{u}(t^{n-1})+5\bm{u}^{n-1}\cdot\nabla\bm{u}^{n-1},\nabla q)\nonumber\\
			&+(5Grg\eta^{n-1}-6Grg\eta^n,\nabla q)+(\nabla p_s(6\bm{e}_u^n-5\bm{e}_u^{n-1}),\nabla q).\nonumber
		\end{align}
		We then begin to prove Theorem \ref{The4.1}. We set \( \bm{v} = -\Delta(6\bm{e}_{\bm{u}}^{n+1} - 5\bm{e}_{\bm{u}}^n) \) in (\ref{4.2}). For the first two terms, we have
		\begin{align}\label{4.14}
			&(11\bm{e}_u^{n+1} - 20\bm{e}_u^n + 9\bm{e}_u^{n-1}, 
			-\Delta(6\bm{e}_u^{n+1} - 5\bm{e}_u^n))\\
			= &\frac{1}{10} (\|\nabla \bm{e}_u^{n+1}\|^2-\|\nabla \bm{e}_u^{n}\|^2 )
			+ \|\frac{9\sqrt{10}}{5} \nabla \bm{e}_u^{n+1}-\frac{\sqrt{90}}{2} \nabla \bm{e}_u^n\|^2 
			- \|\frac{9\sqrt{10}}{5} \nabla \bm{e}_u^{n}-\frac{\sqrt{90}}{2} \nabla \bm{e}_u^{n-1}\|^2  \nonumber\\
			&+\|\frac{\sqrt{90}}{2} \nabla \bm{e}_u^{n+1}-\sqrt{90} \nabla \bm{e}_u^n+\frac{\sqrt{90}}{2} \nabla \bm{e}_u^{n-1}\|^2+\frac{13}{2}\|\nabla(\bm{e}_u^{n+1}-\bm{e}_u^n)\|^2- \frac{9}{2} \|\nabla (\bm{e}_u^n - \bm{e}_u^{n-1})\|^2 \nonumber\\
			&+ \frac{9}{2} \|\nabla (\bm{e}_u^{n+1} - 2\bm{e}_u^n + \bm{e}_u^{n-1})\|^2,
			\nonumber
		\end{align}
		\begin{align}
			&2\delta t \, a\left(5\bm{e}_u^{n+1} - 4\bm{e}_u^n, -\Delta\left(6\bm{e}_u^{n+1} - 5\bm{e}_u^n\right)\right)  \\
			= & \frac{8\delta t}{5}\left\|\Delta\left(6\bm{e}_u^{n+1} - 5\bm{e}_u^n\right)\right\|^2 
			+ \frac{2\delta t}{5}\left\|\Delta\bm{e}_u^{n+1}\right\|^2  + \delta t\left(\left\|\Delta\bm{e}_u^{n+1}\right\|^2 
			- \left\|\Delta\bm{e}_u^{n}\right\|^2 
			+ \left\|\Delta\bm{e}_u^{n+1} - \Delta\bm{e}_u^{n}\right\|^2\right).\nonumber
		\end{align}
		For the nonlinear terms, we use the technique of adding and subtracting a term to split the nonlinearity.
		{ 
		\begin{align}
			&b\left(6\bm{u}^{n} - 5\bm{u}^{n-1}, 6\bm{u}^{n} - 5\bm{u}^{n-1}, \bm{v}\right)- b\left(6\bm{u}(t^{n}) - 5\bm{u}(t^{n-1}), 6\bm{u}(t^{n}) - 5\bm{u}(t^{n-1}), \bm{v}\right)\nonumber\\
			= & b\left(6\bm{e}_u^n - 5\bm{e}_u^{n-1}, \, 6\bm{u}^{n} - 5\bm{u}^{n-1}, \, \bm{v}\right)
			+ b\left(6\bm{u}(t^n) - 5\bm{u}(t^{n-1}), \, 6\bm{e}_u^{n} - 5\bm{e}_u^{n-1}, \, \bm{v}\right).
		\end{align}}
		Then, by using ($\ref{2.5}$) and ($\ref{2.6}$), we can get that
		{ 
		\begin{align}
			&b\left(6\bm{e}_u^n - 5\bm{e}_u^{n-1}, 6\bm{u}^{n} - 5\bm{u}^{n-1}, -\Delta\left(6\bm{e}_u^{n+1} - 5\bm{e}_u^n\right)\right)\\
			&  + b\left(6\bm{u}(t^n) - 5\bm{u}(t^{n-1}), 6\bm{e}_u^{n} - 5\bm{e}_u^{n-1},-\Delta\left(6\bm{e}_u^{n+1} - 5\bm{e}_u^n\right)\right)\notag\\
			\leq & c\bigl\|6\bm{e}_u^n - 5\bm{e}_u^{n-1}\bigr\|_1 
			\bigl\|6\bm{u}^{n} - 5\bm{u}^{n-1}\bigr\|_2 
			\bigl\|\Delta\left(6\bm{e}_u^{n+1} - 5\bm{e}_u^n\right)\bigr\| \notag \\
			&   + c\bigl\|6\bm{u}(t^n) - 5\bm{u}(t^{n-1})\bigr\|_2 
			\bigl\|6\bm{e}_u^{n} - 5\bm{e}_u^{n-1}\bigr\|_1 
			\bigl\|\Delta\left(6\bm{e}_u^{n+1} - 5\bm{e}_u^n\right)\bigr\| \notag \\
			\leq &   c(\varepsilon)\bigl\|\nabla\left(6\bm{e}_u^n - 5\bm{e}_u^{n-1}\right)\bigr\|^2 
			\bigl\|6\bm{u}^{n} - 5\bm{u}^{n-1}\bigr\|_2^2 
			+ \varepsilon\bigl\|\Delta\left(6\bm{e}_u^{n+1} - 5\bm{e}_u^n\right)\bigr\|^2\nonumber\\
			&  +c(\varepsilon)\bigl\|6\bm{u}(t^n) - 5\bm{u}(t^{n-1})\bigr\|_2^2 
			\bigl\|\nabla\left(6\bm{e}_u^{n} - 5\bm{e}_u^{n-1}\right)\bigr\|^2.\nonumber
		\end{align}}
		For the temperature term, by applying Young's inequality, we have
		\begin{align}\label{4.18}
			&(Grg(6\eta^n-5\eta^{n-1}),-\Delta( 6\bm{e}_u^{n+1} - 5\bm{e}_u^n))\nonumber\\
			\leq & c(\varepsilon)\|Grg(6\eta^n-5\eta^{n-1})\|^2+\varepsilon\|\Delta(6 \bm{e}_u^{n+1} - 5\bm{e}_u^n)\|^2.
		\end{align}
		For the pressure terms, using H\"{o}lder's inequality, we get that
		\begin{align}\label{4.19}
			&d(-\Delta( 6\bm{e}_u^{n+1} - 5\bm{e}_u^n),6e^n_p-5e_p^{n-1})\leq\|\nabla(6e^n_p-5e_p^{n-1})\|\|\Delta( 6\bm{e}_u^{n+1} - 5\bm{e}_u^n)\|.
		\end{align}
		We let $q=6e_p^n-5e_p^{n-1}$ in (\ref{4.13}), we can get that
		\begin{align}\label{4.20}
			\|\nabla(6e_p^n-5e_p^{n-1})\|&\leq\|6\eta^n\|+\|5\eta^{n-1}\|+\|\nabla p_s(6\bm{e}_u^n-5\bm{e}_u^{n-1})\|\\
			&\quad+\|6\bm{u}(t^n)\cdot\nabla\bm{u}(t^n)-6\bm{u}^n\cdot\nabla\bm{u}^n\|+{ \|-5\bm{u}^{n-1}\cdot\nabla\bm{u}^{n-1}+5\bm{u}(t^{n-1})\cdot\nabla\bm{u}(t^{n-1})\|}.\nonumber
		\end{align}
		For the right-hand first two term, we  transform it to obtain
		\begin{align}
			6\bm{u}(t^n)\cdot\nabla\bm{u}(t^n)-6\bm{u}^n\cdot\nabla\bm{u}^n
			=-6\bm{e}_u^n\cdot\nabla\bm{u}^n+6\bm{u}(t^n)\cdot\nabla\bm{e}_u^n.
		\end{align}
		Similarly,
		{ 
		\begin{align}
			-5\bm{u}^{n-1}\cdot\nabla\bm{u}^{n-1}+5\bm{u}(t^{n-1})\cdot\nabla\bm{u}(t^{n-1})=5\bm{u}(t^{n-1})\nabla\bm{e}_u^{n-1}-5\bm{e}_u^{n-1}\cdot\nabla\bm{u}^{n-1}.
		\end{align}}
		Then, by using the Sobolev inequality, we get that
		\begin{align}\label{4.23}
			&\|6\bm{u}(t^n)\cdot\nabla\bm{u}(t^n)-6\bm{u}^n\cdot\nabla\bm{u}^n\|^2\\
			\leq&\|6\bm{e}_u^n\cdot\nabla\bm{u}^n\|^2+\|6\bm{u}(t^n)\cdot\nabla\bm{e}_u^n\|^2\leq c\|\Delta\bm{u}^n\|^2\|\nabla\bm{e}_u^n\|^2+c\|\nabla\bm{e}_u^n\|^2\|\Delta\bm{u}(t^n)\|^2,\nonumber
		\end{align}
		where we have used the folllowing Sobolev embedding inequalities
		\begin{align*}
			\|\bm{e}_u^n\cdot\nabla\bm{u}^n\|^2    
			&\leq \|\bm{e}_u^n\|_{L^4}^2 
			\|\nabla\bm{u}^n\|_{L^4}^2
			\leq c\|\nabla\bm{e}_u^n\|^2\|\Delta\bm{u}^n\|^2   \quad d=2,3,\nonumber
		\end{align*}
		\begin{align*}
			&\|\bm{u}(t^n)\cdot\nabla\bm{e}_u^n\|^2\leq\|\bm{u}(t^n)\|_{L^\infty}^2\|\nabla\bm{e}_u^n\|^2 
			\leq\|\Delta\bm{u}(t^n)\|^2\|\nabla\bm{e}_u^n\|^2  \quad d=2,3,
		\end{align*}
		and 
		{ 
		\begin{align*}
			\|-5\bm{u}^{n-1}\cdot\nabla\bm{u}^{n-1}+5\bm{u}(t^{n-1})\cdot\nabla\bm{u}(t^{n-1})\|^2\leq c\|\Delta\bm{u}(t^{n-1})\|^2\|\nabla\bm{e}_u^{n-1}\|^2+ c\|\nabla\bm{e}_u^{n-1}\|^2\|\Delta\bm{u}^{n-1}\|^2.
		\end{align*}}
		Now, by combining (\ref{4.19})–(\ref{4.23}) and making use of Lemma \ref{Lem2.2}, we obtain
		\begin{align}
			&d(-\Delta( 6\bm{e}_u^{n+1} - 5\bm{e}_u^n),6e^n_p-5e_p^{n-1})\\
			\leq &\|\Delta(6\bm{e}_u^{n+1}-5\bm{e}_u^n)\|(\|6\bm{u}(t^n)\cdot\nabla\bm{u}(t^n)-6\bm{u}^n\cdot\nabla\bm{u}^n\|+{ \|5\bm{u}(t^{n-1})\cdot\nabla\bm{u}(t^{n-1})-5\bm{u}^{n-1}\cdot\nabla\bm{u}^{n-1}\|})\nonumber\\
			&  +\|\Delta(6\bm{e}_u^{n+1}-5\bm{e}_u^n)\|(\|6\eta^n\|+\|5\eta^{n-1}\|)+\|\Delta(6\bm{e}_u^{n+1}-5\bm{e}_u^n)\|\|\nabla p_s(6\bm{e}_u^n-5\bm{e}_u^{n-1})\|\nonumber\\
			\leq & c(\varepsilon)(\|6\bm{u}(t^n)\cdot\nabla\bm{u}(t^n)-6\bm{u}^n\cdot\nabla\bm{u}^n\|^2+{ \|5\bm{u}(t^{n-1})\cdot\nabla\bm{u}(t^{n-1})-5\bm{u}^{n-1}\cdot\nabla\bm{u}^{n-1}\|^2})\nonumber\\
			&  +c(\varepsilon)(\|6\eta^n\|^2+\|5\eta^{n-1}\|^2)+\varepsilon\|\Delta(6\bm{e}_u^{n+1}-5\bm{e}_u^n)\|^2+\frac{1}{2}\|\nabla p_s(6\bm{e}_u^n-5\bm{e}_u^{n-1})\|^2+\frac{1}{2}\|\Delta(6\bm{e}_u^{n+1}-5\bm{e}_u^n)\|^2\nonumber\\
			\leq & c(\varepsilon)\|\nabla\bm{e}_u^n\|^2(\|\Delta\bm{u}^n\|^2+\|\bm{u}(t^n)\|_2^2)+c(\varepsilon){ \|\nabla\bm{e}_u^{n-1}\|^2}(\|\Delta\bm{u}^{n-1}\|^2+\|\bm{u}(t^{n-1})\|_2^2)+c(\varepsilon)(\|6\eta^n\|^2+\|5\eta^{n-1}\|^2) \nonumber\\
			&  +(\varepsilon+\frac{1}{2})\|\Delta(6\bm{e}_u^{n+1}-5\bm{e}_u^n)\|^2+(\frac{1}{4}+\varepsilon)\|\Delta(6\bm{e}_u^{n}-5\bm{e}_u^{n-1})\|^2+C(\varepsilon)\|\nabla(6\bm{e}_u^n-5\bm{e}_u^{n-1})\|^2.\nonumber
		\end{align}
			{ 
	We now estimate the terms on the right-hand side. Using Taylor expansions, Young's inequality and the Cauchy–Schwarz inequatility, we deduce that 
	\begin{align} \label{4.25}
		&(B^n,-\Delta(6\bm{e}_u^{n+1}-5\bm{e}_u^n))\\
		= &(6\int^{t^{n+5}}_{t^n}(t^n-s)\nabla\frac{\partial^2p}{\partial t^2}(s)ds-5\int^{t^{n+5}}_{t^{n-1}}(t^{n-1}-s)\nabla\frac{\partial^2p}{\partial t^2}(s)ds,-\Delta(6\bm{e}_u^{n+1}-5\bm{e}_u^n))\nonumber\\
		\leq&c(\varepsilon)\|6\int^{t^{n+5}}_{t^n}(t^n-s)\nabla\frac{\partial^2p}{\partial t^2}(s)ds-5\int^{t^{n+5}}_{t^{n-1}}(t^{n-1}-s)\nabla\frac{\partial^2p}{\partial t^2}(s)ds\|^2+\varepsilon \|\Delta(6\bm{e}_u^{n+1}-5\bm{e}_u^n)\|^2\nonumber\\
		\leq&c(\varepsilon)\mid\int^{t^{n+5}}_{t^{n}}(t^n-s)^2ds+\int^{t^{n+5}}_{t^{n-1}}(t^{n-1}-s)^2ds\mid\int^{t^{n+5}}_{t^{n-1}}\|\nabla\frac{\partial^2p}{\partial t^2}(s)\|^2+\varepsilon \|\Delta(6\bm{e}_u^{n+1}-5\bm{e}_u^n)\|^2\nonumber\\
		\leq & c(\varepsilon)\delta t^3\int^{t^{n+5}}_{t^{n-1}}\|\nabla\frac{\partial^2p}{\partial t^2}(s)\|^2ds+\varepsilon \|\Delta(6\bm{e}_u^{n+1}-5\bm{e}_u^n)\|^2,\nonumber
	\end{align}}
		\begin{align}
			&(Q^n,-\Delta(6\bm{e}_u^{n+1}-5\bm{e}_u^n))\\
			= &(
			-5\int^{t^{n+5}}_{t^{n+1}}(t^{n+1}-s)\Delta\frac{\partial^2\bm{u}}{\partial t^2}(s)ds+4\int^{t^{n+5}}_{t^{n}}(t^{n+1}-s)\Delta\frac{\partial^2\bm{u}}{\partial t^2}(s)ds,-\Delta(6\bm{e}_u^{n+1}-5\bm{e}_u^n))\nonumber\\
			\leq & c(\varepsilon)\delta t^3\int^{t^{n+5}}_{t^{n}}\|\Delta\frac{\partial^2\bm{u}}{\partial t^2}(s)\|^2ds+\varepsilon\|\Delta(6\bm{e}_u^{n+1}-5\bm{e}_u^n)\|^2,
		\end{align}
		\begin{align}
			&(R^n,-\Delta(6\bm{e}_u^{n+1}-5\bm{e}_u^n))\nonumber\\
			\leq &\frac{c(\varepsilon)}{\delta t}\|\frac{11}{2}\int^{t^{n+5}}_{t^{n+1}}(t^{n+1}-s)^2\frac{\partial^3\bm{u}}{\partial t^3}(s)ds-10\int^{t^{n+5}}_{t^{n}}(t^{n}-s)^2\frac{\partial^3\bm{u}}{\partial t^3}(s)ds\nonumber\\
			&+\frac{9}{2}\int^{t^{n+5}}_{t^{n-1}}(t^{n-1}-s)^2\frac{\partial^3\bm{u}}{\partial t^3}(s)ds\|^2  +\varepsilon\delta t\|\Delta(6\bm{e}_u^{n+1}-5\bm{e}_u^n)\|^2\nonumber\\
			\leq & c(\varepsilon)\delta t^4\int^{t^{n+5}}_{t^{n-1}}\|\frac{\partial^3\bm{u}}{\partial t^3}(s)\|^2ds+\varepsilon\delta t\|\Delta(6\bm{e}_u^{n+1}-5\bm{e}_u^n)\|^2,
		\end{align}
		\begin{align}
			&(H^n,-\Delta(6\bm{e}_u^{n+1}-5\bm{e}_u^n)) \\
			= &(6\int^{t^{n+5}}_{t^{n}}(t^{n}-s)\frac{\partial^2\theta}{\partial t^2}(s)ds-5\int^{t^{n+5}}_{t^{n-1}}(t^{n-1}-s)\frac{\partial^2\theta}{\partial t^2}(s)ds,\Delta(6\bm{e}_u^{n+1}-5\bm{e}_u^n))\nonumber\\
			\leq & c(\varepsilon)\delta t^3\int^{t^{n+5}}_{t^{n-1}}\|\frac{\partial^2\theta}{\partial t^2}(s)\|^2ds+\varepsilon\|\Delta(6\bm{e}_u^{n+1}-5\bm{e}_u^n)\|^2.\nonumber
		\end{align}
		By using the (\ref{2.5}) and (\ref{4.5}) , for the term $S^n$ that
		\begin{align}\label{4.30}
			&(S^n,-\Delta(6\bm{e}_u^{n+1}-5\bm{e}_u^n))\\
			=&({ \bm{u}(t^{n+5})\cdot\nabla(\bm{u}(t^{n+5})-(6\bm{u}(t^{n})-5\bm{u}(t^{n-1}))},-\Delta(6\bm{e}_u^{n+1}-5\bm{e}_u^n))\nonumber\\
			& -({ 6\bm{u}(t^{n})-5\bm{u}(t^{n-1})-\bm{u}(t^{n+5}))\cdot\nabla(6\bm{u}(t^{n})-5\bm{u}(t^{n-1}))},-\Delta(6\bm{e}_u^{n+1}-5\bm{e}_u^n))\nonumber\\
			\leq& C\|\bm{u}(t^{n+5})\|_2{ \|\nabla(\bm{u}(t^{n+5})-6\bm{u}(t^{n})+5\bm{u}(t^{n-1}))\|}\|\Delta(6\bm{e}_u^{n+1}-5\mathbf{e}_u^n)\|\nonumber\\
			& +C{ \|6\bm{u}(t^{n})-5\bm{u}(t^{n-1})\|_2}\|\nabla(\bm{u}(t^{n+5})-6\bm{u}(t^{n})+5\bm{u}(t^{n-1}))\|\|\Delta(6\bm{e}_u^{n+1}-5\bm{e}_u^n)\|\nonumber\\
			\leq& C(\|\bm{u}(t^{n+5})\|^2_2+{ \|6\bm{u}(t^{n})-5\bm{u}(t^{n-1})\|_2^2})\delta t^3\int^{t^{n+5}}_{t^{n-1}}\|\nabla\frac{\partial^2\bm{u}}{\partial t^2}(s)\|^2ds+\varepsilon\|\Delta(6\bm{e}_u^{n+1}-5\bm{e}_u^n)\|^2\nonumber\\
			\leq& C\delta t^3\int^{t^{n+5}}_{t^{n-1}}\|\nabla\frac{\partial^2\bm{u}}{\partial t^2}(s)\|^2ds+\varepsilon\|\Delta(6\bm{e}_u^{n+1}-5\bm{e}_u^n)\|^2.\nonumber
		\end{align}
		Combining (\ref{4.14}) to (\ref{4.30}), to ensure the stability we can choose $\varepsilon$ enough small such that $\frac{4}{5}>6\varepsilon+\frac{3}{4}$ and drop some unnecessary terms, we get
		\begin{align}
			& \frac{1}{10} \Bigl( \|\nabla \boldsymbol{e}_u^{n+1}\|^2 - \|\nabla \boldsymbol{e}_u^n\|^2 \Bigr) 
			+  \Bigl\|\frac{9\sqrt{10}}{5} \nabla \boldsymbol{e}_u^{n+1} - \frac{\sqrt{90}}{2} \nabla \boldsymbol{e}_u^n \Bigr\|^2 
			+ \frac{13}{2} \|\nabla (\boldsymbol{e}_u^{n+1} - \boldsymbol{e}_u^{n})\|^2 \nonumber\\
			&- \frac{9}{2} \|\nabla (\boldsymbol{e}_u^{n} - \boldsymbol{e}_u^{n-1})\|^2 
			+ \frac{8\delta t}{5} \|\Delta (6\boldsymbol{e}_u^{n+1} - 5\boldsymbol{e}_u^n)\|^2 
			+ \frac{2\delta t}{5} \|\Delta \boldsymbol{e}_u^{n+1}\|^2+ \delta t \Bigl( \|\Delta \boldsymbol{e}_u^{n+1}\|^2 - \|\Delta \boldsymbol{e}_u^n\|^2 \Bigr)\nonumber\\
			\leq &C \delta t\bigl\|6\bm{u}^n - 5\bm{u}^{n-1}\bigr\|_2^2
			{ \bigl\|\nabla\bigl(6\bm{e}_u^{n} - 5 \bm{e}_u^{n-1}\bigr)\bigr\|^2 +C\delta t\|\nabla(6\bm{e}_u^{n}-5\bm{e}_u^{n-1})\|^2}+C\delta t\|Grg(6\eta^n-5\eta^{n-1})\|^2\nonumber\\
			&+C\delta t \|\nabla\bm{e}_u^n\|^2(\|\Delta\bm{u}^n\|^2+\|\bm{u}(t^n)\|_2^2)+C\delta t{ \|\nabla\bm{e}_u^{n-1}\|^2}(\|\Delta\bm{u}^{n-1}\|^2+\|\bm{u}(t^{n-1})\|_2^2) \nonumber\\
			&+C\delta t^4\int^{t^{n+5}}_{t^{n-1}}\|\frac{\partial^3\bm{u}}{\partial t^3}\|^2+\|\nabla\frac{\partial^2 p}{\partial t^2}\|^2+\|\Delta\frac{\partial^2\bm{u}}{\partial t^2}\|^2+\|\nabla\frac{\partial^2\bm{u}}{\partial t^2}\|^2+\|\frac{\partial^2\theta}{\partial t^2}\|^2 ds. \label{4.31}
		\end{align}
		Next, by summing over \( n \) from \( 1 \) to \( m \) in (\ref{4.31}) and dropping some unnecessary terms, we get
		\begin{align}
			&\|\nabla\bm{e}_u^{m+1}\|^2+\delta t\sum_{n=0}^{m+1}\|\Delta\bm{e}_u^n\|^2\nonumber\\
			\leq& C\delta t\sum_{n=0}^{m}(\|\Delta\bm{u}^n\|^2+\|\Delta\bm{u}(t^n)\|^2)\|{ \nabla\bm{e}_u^n\|^2}+C\delta t\sum_{n=0}^{m}\|\eta^n\|^2+CT\delta t^4. \label{4.32}
		\end{align}
		Then, we take \( \bm{w} = -\Delta(6\eta^{n+1} - 5\eta^n) \) in (\ref{4.8}). For the first two terms, we obtain
		\begin{align} \label{4.33}
			&(11\eta^{n+1} - 20\eta^n + 9\eta^{n-1}, 
			-\Delta(6\eta^{n+1} - 5\eta^n))\\
			=& \frac{1}{10} (\|\nabla \eta^{n+1}\|^2-\|\nabla \eta^{n}\|^2 )
			+ \|\frac{9\sqrt{10}}{5} \nabla \eta^{n+1}-\frac{\sqrt{90}}{2} \nabla \eta^n\|^2 - \|\frac{9\sqrt{10}}{5} \nabla \eta^{n}-\frac{\sqrt{90}}{2} \nabla \eta^{n-1}\|^2  \nonumber\\
			&+\|\frac{\sqrt{90}}{2} \nabla \eta^{n+1}-\sqrt{90} \nabla \eta^n+\frac{\sqrt{90}}{2} \nabla \eta^{n-1}\|^2+\frac{13}{2}\|\nabla(\eta^{n+1}-\eta^n)\|^2- \frac{9}{2} \|\nabla (\eta^n - \eta^{n-1})\|^2 \nonumber\\
			&
			+ \frac{9}{2} \|\nabla (\eta^{n+1} - 2\eta^n + \eta^{n-1})\|^2 ,
			\nonumber
		\end{align}
		\begin{align}
			&2\delta t \,a\left(5\eta^{n+1}-4\eta^n,-\Delta( 6\eta^{n+1} - 5\eta^n)\right)\\
			= &\frac{8\delta t}{5}\|\Delta( \eta^{n+1} - 5\eta^n)\|^2+\frac{2\delta t}{5}\|\Delta\eta^{n+1}\|^2+\delta t(\|\Delta\eta^{n+1}\|^2-\|\Delta\eta^{n}\|^2+\|\Delta\eta^{n+1}-\Delta\eta^{n}\|^2).\nonumber
		\end{align}
		For the nonlinear terms, we deduce that
		{ 
		\begin{align}
			&b(6\bm{u}^n-5\bm{u}^{n-1},6\theta^{n}-5\theta^{n-1},\bm{w})-b(6\bm{u}(t^n)-5\bm{u}(t^{n-1}),6\theta(t^{n})-5\theta(t^{n-1}),\bm{w})\nonumber\\
			= &b(6\bm{e}^n_u-5\bm{e}^{n-1}_u,6\theta(t^{n})-5\theta(t^{n-1}),\bm{w})+b(6\bm{u}(t^n)-5\bm{u}(t^{n-1}),6\eta^{n}-5\eta^{n-1},\bm{w}).
		\end{align}}
		Hence, from (\ref{2.5}), we get
		{ 
		\begin{align}
			&b(6\bm{u}(t^n)-5\bm{u}(t^{n-1}),6\eta^{n}-5\eta^{n-1},-\Delta(6 \eta^{n+1} - 5\eta^n)\nonumber\\
			&\quad+b(6\bm{e}^n_u-5\bm{e}^{n-1}_u,6\theta(t^{n})-5\theta(t^{n-1}),-\Delta(6 \eta^{n+1} - 5\eta^n))\nonumber\\
			&\leq c(\varepsilon) \|\Delta(6\bm{u}(t^n)-5\bm{u}(t^{n-1}))\|^2\|\nabla(6\eta^{n}-5\eta^{n-1})\|^2\nonumber\\
			& \quad+c(\varepsilon) \|\Delta(6\theta(t^{n})-5\theta(t^{n-1}))\|^2\|\nabla(6\bm{e}_u^{n}-5\bm{e}_u^{n-1})\|^2+\varepsilon\|\Delta(6 \eta^{n+1} - 5\eta^n)\|^2.
		\end{align}}
		As for the right-hand side terms of the equation, they can be estimated in the same manner as in (\ref{4.25})-(\ref{4.30}).  
		\begin{align}
			&(T^n,-\Delta(6 \eta^{n+1} - 5\eta^n))\leq c(\varepsilon)\delta t^4\int_{t^{n-1}}^{t^{n+5}}\left\|\frac{\partial^3\theta}{\partial t^3}\right\|^2ds+\varepsilon \delta t\|\Delta(6 \eta^{n+1} - 5\eta^n)\|^2,
		\end{align}
		\begin{align}
			(W^n,-\Delta(6 \eta^{n+1} - 5\eta^n))\leq c(\varepsilon)\delta t^3\int_{t^{n}}^{t^{n+5}}\left\|\Delta\frac{\partial^2\theta}{\partial t^2}\right\|^2ds+\varepsilon \|\Delta(6 \eta^{n+1} - 5\eta^n)\|^2,
		\end{align}
		\begin{align}\label{4.39}
			&( E^n, -\Delta\bigl(6\eta^{n+1} - 5\eta^n\bigr) )
			\leq c(\varepsilon)\delta t^3\int_{t^{n-1}}^{t^{n+5}}\left(\left\|
			\nabla\frac{\partial^2\theta}{\partial t^2}\right\|^2+\left\|
			\frac{\partial^2\bm{u}}{\partial t^2}\right\|^2\right) ds 
			+ \varepsilon\|\Delta\bigl(6\eta^{n+1} - 5\eta^n\bigr)\|^2.
		\end{align}
		Combining (\ref{4.33})-(\ref{4.39}), and we choose $\varepsilon$ small enough such that $\frac{8}{5}>4\varepsilon$ , dropping some unnecessary terms, we get:
		\begin{align}
			& \frac{1}{10} \left( \|\nabla \eta^{n+1}\|^2 - \|\nabla \eta^n\|^2 \right) 
			+ \left\| \frac{9\sqrt{10}}{5}\nabla \eta^{n+1} - \frac{3\sqrt{10}}{2} \nabla \eta^n \right\|^2 
			- \left\| \frac{9\sqrt{10}}{5}\nabla \eta^{n} - \frac{3\sqrt{10}}{2} \nabla \eta^{n-1}\right\|^2 \nonumber\\
			& + \frac{13}{2} \|\nabla (\eta^{n+1} - \eta^n)\|^2 - \frac{9}{2} \|\nabla (\eta^{n} - \eta^{n-1})\|^2 
			+ \frac{8\delta t}{5} \|\Delta (6\eta^{n+1} - 5\eta^n)\|^2 
			+ \frac{2\delta t}{5} \|\Delta \eta^{n+1}\|^2 \nonumber\\
			& + \delta t \left( \|\Delta \eta^{n+1}\|^2 - \|\Delta \eta^n\|^2 \right) \nonumber\\
			\leq & c(\varepsilon) \delta t \left\| 6\bm{u}(t^n) - 5\bm{u}(t^{n-1}) \right\|_2^2 { \left\| \nabla (6\eta^{n} - 5\eta^{n-1}) \right\|^2 }
			+ c(\varepsilon) \delta t \left\| 6\theta(t^{n}) - 5\theta(t^{n-1}) \right\|_2^2 { \left\| \nabla (6\bm{e}_u^{n} - 5\bm{e}_u^{n-1}) \right\|^2} \nonumber\\
			& + C\delta t^4 \int_{t^{n-1}}^{t^{n+5}} \left( 
			\left\| \frac{\partial^3\theta}{\partial t^3} \right\|^2 
			+ \left\| \Delta \frac{\partial^2\theta}{\partial t^2} \right\|^2 
			+ \left\| \nabla \frac{\partial^2\theta}{\partial t^2} \right\|^2 
			+ \left\| \frac{\partial^2\bm{u}}{\partial t^2} \right\|^2 
			\right) ds. \label{4.40}
		\end{align}
		By summing in (\ref{4.40}) over \( n \) from \( 1 \) to \( m \) and setting \( \varepsilon = \frac{1}{100} \), we obtain
		\begin{align}
			&\|\nabla\eta^{m+1}\|^2+\delta t\sum_{n=0}^{m+1}\|\Delta\eta^n\|^2 \leq C\delta t\sum_{n=0}^{m}\|\Delta\bm{u}(t^n)\|^2{ \|\nabla\eta^n\|^2}+ C\delta t\sum_{n=0}^{m}\|\Delta\theta(t^n)\|^2\|\nabla\bm{e}_u^n\|^2+CT\delta t^4. \label{4.41}
		\end{align}
		By taking the sum in (\ref{4.32}) and (\ref{4.41}), we obtain
		\begin{align}\label{4.46}
			&\|\nabla\eta^{m+1}\|^2+\|\nabla\bm{e}_u^{m+1}\|^2+\delta t\sum_{n=0}^{m+1}(\|\Delta \eta^n\|^2+\|\Delta\bm{e}_u^n\|^2)\\
			&\leq C\delta t\sum_{n=0}^{m}(\|\Delta\bm{u}^n\|^2+\|\Delta\bm{u}(t^n)\|^2+\|\Delta\theta(t^n)\|^2)({ \|\nabla\bm{e}_u^n\|^2+\|\nabla\eta^n\|^2})+C\delta t^4.\nonumber
		\end{align}
		Finally, by applying Lemma \ref{Lem2.3} to (\ref{4.46}), we obtain
		\begin{align}
			\|\nabla\eta^{m+1}\|^2+\|\nabla\bm{e}_u^{m+1}\|^2+\delta t\sum_{n=0}^{m+1}(\|\Delta\eta^n\|^2+\|\Delta\bm{e}_u^n\|^2)\leq C\delta t^4.
		\end{align}
		Next, from (\ref{4.20}) and using the same approach as in (\ref{4.23}), we derive
		\begin{align}
			\|\nabla e_p^n\|^2 \leq& c\|\nabla\bm{e}_u^n\|^2\left(\|\Delta\bm{u}^n\|^2 + c\|\bm{u}(t^n)\|_2^2\right) + \|\nabla\bm{e}_u^{n-1}\|^2\left(\|\Delta\bm{u}^{n-1}\|^2 + c\|\bm{u}(t^{n-1})\|_2^2\right)\notag\\
			&+ \|6\eta^n\| + \|5\eta^{n-1}\| + \|\nabla p_s(6\bm{e}_u^n - 5\bm{e}_u^{n-1})\|. \label{4.44}
		\end{align}
		Then, by taking the sum in (\ref{4.44}) over \( n \) from \( 1 \) to \( m \), we obtain that
		\begin{align}
			\delta t\sum_{n=0}^{m+1}\|\nabla e_p^n\|^2 \leq C \delta t^4,
		\end{align}
		where the constant \( C \) is independent of \( \delta t \) and \( n \), but may depend on known quantities such as \( T \), \( \Omega \), and the exact solutions. This completes the proof.
	\end{proof}
	\section{Numerical results}
\setcounter{equation}{0}
In this section, we present numerical examples to demonstrate the performance of the new second-order consistent splitting scheme, with implementations carried out using FreeFEM++; we consider the natural convection equations (\ref{1.1}) defined on the spatiotemporal domain \([0,T] \times \Omega = [0,1] \times [0,1]^2\), where a no-slip boundary condition is imposed on the velocity \(\bm{u}\). To quantitatively validate the scheme's accuracy, we adopt the following trigonometric exact solutions: 
\begin{align*} 
	\bm{u}_1(x,y,t) = & \sin(2\pi x + t)\sin(2\pi y + t), \\ \bm{u}_2(x,y,t) = & \cos(2\pi x + t)\cos(2\pi y + t), \\ p(x,y,t) = & \sin(2\pi(x + y) + t), \\ \theta(x,y,t) = & \sin(2\pi x + t)\sin(2\pi y + t).
\end{align*}
 The body force, heat sources, and boundary conditions for the natural convection equations are derived directly from these exact solutions, which ensures the numerical results can be compared against a known reference and eliminates ambiguity in error quantification. We use the \(P1b$-$P1$-$P1b\) element pair to approximate velocity (\(\bm{u}\)), pressure (\(p\)), and temperature (\(\theta\)) respectively, and these finite element spaces satisfy the LBB (Ladyzhenskaya-Babuška-Brezzi) condition-a critical requirement to avoid pressure oscillations and ensure stable coupling between velocity and pressure. We set the time step as \(\delta t = 1/(10n)\), where \(n\) is a loop variable ranging over \(n = 1, 2, \dots, N\) (with \(N\) denoting the spatial mesh refinement parameter and specific values set as \(N = 8, 16, 24, 32,40,48\) to cover 8 refinement levels for convergence analysis), and the spatial mesh size is \(h = 1/N\), consistent with the refinement parameter \(N\) in the time step definition to ensure synchronized spatiotemporal refinement. To assess the numerical accuracy of the scheme, we compute the following error norms (quantifying the deviation between numerical and exact solutions): for velocity \(\bm{u}\), we calculate the \(L^2\)-norm error (measuring overall function deviation) and \(H^1\)-norm error (measuring deviations in both function values and their gradients); for temperature \(\theta\), we also compute the \(L^2\)-norm error and \(H^1\)-norm error (consistent with velocity to maintain uniform accuracy assessment); for pressure \(p\), we only compute the \(L^2\)-norm error-a standard choice for pressure error evaluation in incompressible flow problems, as pressure gradients (rather than absolute values) dominate flow dynamics.
\begin{table}[htbp]
	\centering 
	\caption{The numerical results with $\nu=1$, $\lambda=1$, $Gr=1$ and $T=1$.}\label{T1}
	\begin{tabular}{lllllllll}
		\toprule
		$1/h$ & $\dfrac{\| \boldsymbol{u} - \boldsymbol{u}_h \|_0}{\| \boldsymbol{u} \|_0}$ & $\dfrac{\| \boldsymbol{u} - \boldsymbol{u}_h \|_1}{\| \boldsymbol{u} \|_1}$ & $\dfrac{\| p - p_h \|_0}{\| p \|_0}$ & $\dfrac{\| \theta - \theta_h \|_0}{\| \theta \|_0}$ & $\dfrac{\| \theta - \theta_h \|_1}{\| \theta \|_1}$  & CPU(s) \\
		\midrule
	8  & 0.302232 & 0.401152 & 0.107112 & 0.298551 & 0.400572 & 4.573\\
16  & 0.0945844 & 0.193611 & 0.0374296 & 0.0938261 & 0.193494 & 38.953 \\
	24  & 0.0446914 & 0.126788 & 0.0183279 & 0.0442889 & 0.126743 & 155.397\\
32 & 0.0256652 & 0.0943124 & 0.0106265 & 0.0254507 & 0.0942928 & 427.721\\
40  & 0.0166122 & 0.0751317 & 0.00691883 & 0.0164705 & 0.075121 & 
 967.995\\
	48  & 0.0116094 & 0.062458 & 0.00485108 & 0.0115088 & 0.0624516 & 9718.53 \\
	56  & 0.0085737 & 0.0534549 & 0.00359346 & 0.00849536 & 0.0534507 &  4790.05 \\
64  & 0.00658027 & 0.0467262 & 0.00276125 & 0.00651992 & 0.0467233 & 7256.67 \\
		\bottomrule
	\end{tabular}
	\label{tab:nu1}
\end{table}

\begin{table}[htbp]
	\centering
	\caption{Convergence orders with $\nu=1$, $\lambda=1$, $Gr=1$ and $T=1$.}
	\label{T2}
	\begin{tabular*}{0.66\textwidth}{llllll
		}
		\toprule
		$1/h$ & $u_L^2$-order & $u_H^1$-order & $p_L^2$-order &  $\omega_L^2$-order & $\omega_H^1$-order \\
		\midrule
		8   &--- & ---   &---   & ---  & ---  \\
	 $16$& 1.6760 & 1.0510 & 1.5169 & 1.6699 & 1.0498 \\
	  $24$& 1.8490 & 1.0441 & 1.7610 & 1.8450 & 1.0435 \\
	  $32$& 1.9280 & 1.0286 & 1.8947 & 1.9257 & 1.0281 \\
	  $40$& 1.9494 & 1.0186 & 1.9230 & 1.9502 & 1.0186 \\
	  $48$& 1.9653 & 1.0133 & 1.9474 & 1.9630 & 1.0131 \\
	  $56$& 1.9742 & 1.0098 & 1.9467 & 1.9720 & 1.0096 \\
	  $64$& 1.9817 & 1.0075 & 1.9728 & 1.9820 & 1.0074 \\
		\bottomrule
	\end{tabular*}
\end{table}

\begin{table}[htbp]
	\centering
	\caption{The numerical results with $\nu=0.01$, $\lambda=1$, $Gr=1$ and $T=1$.} \label{T3}
	\begin{tabular}{lllllllll}
		\toprule
		$1/h$ & $\dfrac{\| \boldsymbol{u} - \boldsymbol{u}_h \|_0}{\| \boldsymbol{u} \|_0}$ & $\dfrac{\| \boldsymbol{u} - \boldsymbol{u}_h \|_1}{\| \boldsymbol{u} \|_1}$ & $\dfrac{\| p - p_h \|_0}{\| p \|_0}$ & $\dfrac{\| \theta - \theta_h \|_0}{\| \theta \|_0}$ & $\dfrac{\| \theta - \theta_h \|_1}{\| \theta \|_1}$ &  CPU(s) \\
		\midrule
	8 & 0.727041 & 1.28491 & 0.190067 & 0.807644 & 1.28692 & 2.632 \\
	16 & 0.340417 & 0.49406 & 0.134762 & 0.341902 & 0.498408 & 37.351 \\
	24 & 0.189821 & 0.284953 & 0.0902497 & 0.205459 & 0.310247 & 154.921 \\
	32 & 0.114041 & 0.182267 & 0.0597876 & 0.115895 & 0.18973 & 435.588 \\
	40 & 0.0769314 & 0.130048 & 0.0423899 & 0.0770733 & 0.134026 & 1075.78 \\
	48 & 0.055268 & 0.0989118 & 0.0313759 & 0.0548273 & 0.101176 & 2094.73 \\
	56 & 0.041986 & 0.0792107 & 0.0242714 & 0.0421359 & 0.081463 & 4057.15 \\
	64 & 0.0326544 & 0.0652822 & 0.0191332 & 0.0325573 & 0.0667543 & 7226.94 \\
\hline
		\bottomrule
	\end{tabular}
	\label{tab:nu0}
\end{table}

\begin{table}[htbp]
	\centering
	\caption{Convergence orders with $\nu=0.01$, $\lambda=1$, $Gr=1$ and $T=1$.}
	\label{T4}
	\begin{tabular*}{0.66\textwidth}{llllll
		}
		\toprule
		$1/h$ & $u_L^2$-order & $u_H^1$-order & $p_L^2$-order &  $\theta_L^2$-order & $\theta_H^1$-order \\
		\midrule
		8   &--- & ---   &---   & ---  & ---  \\
		$16$ & 1.0941 & 1.3797 & 0.4960 & 1.2417 & 1.3707 \\
		$24$ & 1.5370 & 1.3437 & 0.8641 & 1.3917 & 1.2823 \\
		$32$ & 1.7901 & 1.4033 & 1.2248 & 1.7307 & 1.3610 \\
		$40$ & 1.7770 & 1.3762 & 1.4247 & 1.7898 & 1.3400 \\
		$48$ & 1.8044 & 1.3164 & 1.5104 & 1.8427 & 1.2903 \\
		$56$ & 1.8340 & 1.2492 & 1.5500 & 1.7884 & 1.2198 \\
		$64$ & 1.8744 & 1.2072 & 1.5780 & 1.8792 & 1.1811 \\
		\bottomrule
	\end{tabular*}
\end{table}

Table \ref{T1} and Table \ref{T3} present the error behaviors and computational efficiency of the numerical solutions with mesh refinement under different viscosity coefficients \(\nu\).  With mesh refinement, errors gradually diminish and achieve optimal convergence rates. The convergence order is shown in Table \ref{T2} and Table \ref{T4}, and the \(L^2\)-norm convergence order is nearly theoretically optimal order.  Furthermore, increasing the viscosity coefficient \(\nu\) does not change the convergence trend of errors. Notably, for the same mesh, the absolute errors when \(\nu=1\) are generally smaller than those when \(\nu=0.01\). This is a common phenomenon in incompressible fluid dynamics: the higher the viscosity, the more stable the flow pattern.

	{ 
	
	\section{Conclusion}
	
	In this paper, we developed a new  second-order consistent splitting schemes for the natural convection equations subject to no-slip boundary conditions. The main idea is to introduce a positive integer parameter \(k\) and construct the schemes through Taylor expansions about the shifted time level \(t^{n+k}\). 
	
	This framework generalizes the conventional construction based on expansions about \(t^{n+1}\). The stability analysis shows that the coefficient condition required by the energy argument is satisfied when \(k>4\). Accordingly, we selected \(k=5\), the smallest admissible integer, and established the stability and convergence properties of the resulting scheme. Under suitable regularity assumptions, stability estimates and error bounds were derived in both two and threedimensional settings.
	
	The numerical experiments are consistent with the theoretical results and confirm the expected second-order temporal convergence. They also demonstrate that the proposed splitting scheme provides an accurate and computationally effective approach for solving the natural convection equations.					
	
	Future work will address long-time and global stability, examine how different choices of \(k\) affect accuracy, stability, and computational cost, and extend the shifted-expansion framework to higher-order consistent splitting schemes for natural convection problems.

}

\end{document}